\documentclass[reqno]{article}
\usepackage[ansinew]{inputenc}
\usepackage{latexsym}
\usepackage{amsfonts}
\usepackage{amsmath,mathrsfs,amssymb,amsthm,amscd}
\usepackage{longtable}
\allowdisplaybreaks
\usepackage{multirow}
\usepackage{color}
\usepackage{url}
\usepackage{paralist}
\usepackage{cite}
\newtheorem{theo}{Th\'eor\`eme}[subsection]
\newtheorem{lem}[theo]{Lemme}
\newtheorem{prop}[theo]{Proposition}

\theoremstyle{definition}
\newtheorem{defi}[theo]{D\'efinition}
\newtheorem{exe}[theo]{Exemple}
\theoremstyle{remark}

\newcommand{\E}{\mathcal{E}}

\newcommand{\I}{\mathcal{I}}
\newcommand{\J}{\mathcal{J}}
\newcommand{\g}{\gamma}
\newcommand{\al}{\alpha}

\newcommand{\de}{\delta}
\newcommand{\be}{\beta}
\newcommand{\D}{\mathfrak D}
\newcommand{\Le}{\mathfrak L}
\newcommand{\In}{\mathfrak I}

\newcommand{\ann}{\textrm{ ann}}
\newcommand{\Hom}{\textrm{Hom}}
\def\thebibliography#1{\section*{{\large Bibliographie}\markboth
 {REFERENCES}{REFERENCES}}\list
 {[\arabic{enumi}]}{\settowidth\labelwidth{[#1]}\leftmargin\labelwidth
 \advance\leftmargin\labelsep
 \usecounter{enumi}}
 \def\newblock{\hskip .11em plus .33em minus -.07em}
 \sloppy
 \sfcode`\.=1000\relax}

\def\qed{\relax\ifmmode\hskip2em \Box\else\unskip\nobreak\hskip1em \hfill$\Box$\fi}

\title{D\'{e}rivations dans les alg\`{e}bres d'\'{e}volution \`{a} puissances associatives}

\author{
Moussa OUATTARA\thanks{\texttt{ouatt\_ken@yahoo.fr}}\quad\quad
Souleymane SAVADOGO\thanks{\texttt{sara01souley@yahoo.fr}}\\
D\'epartement de Math\'ematiques et Informatique\\
Universit\'e Ouaga I Pr Joseph KI-ZERBO\\
03 BP 7021 Ouagadougou 03\\
Burkina Faso
}
\date{}
\begin{document}
\maketitle

\begin{abstract} In this paper, we investigate the derivations in evolution algebras that are power-associative. This problem is reduced to that of power-associative evolution nilalgebras. We show how to calculate derivations in decomposable algebras. This caculation shows that it is enough to describe derivations in indecomposable evolution algebras. We first determine the derivation algebra of $n$-dimensional indecomposable associative evolution nilalgebras with one-dimensional annihilator.
We describe the derivation algebra of indecomposable nilalgebras, up to dimension $6$, that are associative or not.
In each cases, we give the commutator of two derivations.  We also describe the ideal of inner derivation.

\medskip 2010 Mathematics Subject Classification~: Primary 17D92, 17A05, Secondary 17D99, 17A60

\end{abstract}
\bigskip\bigskip

\section{Introduction}
Il existe plusieurs classes d'alg\`{e}bres non-associatives : les
alg\`{e}bres train, les alg\`{e}bres de Jordan, les alg\`ebres alternatives,
les alg\`{e}bres de Bernstein \ldots

Dans \cite{Tian2006}, les Auteurs
d\'{e}finissent une nouvelle classe d'alg\`{e}bres non-associatives, celle
des alg\`{e}bres d'\'{e}volution. Cette classe d'alg\`{e}bre ne forme pas une
variet\'{e}, i.e. elle n'est pas d\'{e}finie par des identit\'{e}s ; par
cons\'{e}quent, l'\'{e}tude de ces alg\`{e}bres suit des voies différentes
\cite{Casado,Mercedes2016,camacho,Casas2014,Elduque2015,Elduque2016,Hegazi2015,Labra,Tian2006}.

Dans \cite{Tian2008}, l'Auteur montre qu'une alg\`{e}bre d'\'{e}volution est commutative (donc
flexible), qu'elle n'est pas n\'{e}cessairement associative, ni \`{a}
puissances associatives. Toutefois, dans \cite{Ouatt2018}, les Auteurs
caract\'erise les alg\`{e}bres d'\'{e}volution qui sont associatives, de même
que celles qui sont \`{a} puissances associatives. Ils montrent notamment que
toute alg\`{e}bre d'\'{e}volution qui est \`{a} puissances associatives est
une alg\`ebre de Jordan. La d\'ecomposition de Wedderburn des alg\`ebres
d'\'evolution qui son \`a puissances associatives est construite.

\textit{Une d\'{e}rivation} dans une alg\`{e}bre $\E$ sur un corps $K$ est une
application lin\'{e}aire $d~: \E \longrightarrow \E$ v\'{e}rifiant $d(uv) =
ud(v) + d(u)v$ pour tous $u, v \in \E$. Il est bien connu que l'espace
$\D(\E)$ de toutes les d\'{e}rivations de $\E$ est une alg\`{e}bre de
Lie, appel\'ee  \textit{alg\`ebre de Lie des d\'erivations}, o\`u le crochet
de deux d\'erivations $d$ et $d'$ est d\'{e}fini par $[d, d'] = d\circ d' - d'\circ d$. On note $\D(\E)'$ son alg\`ebre d\'eriv\'ee
\cite{N. Jacobson}.

Dans la th\'{e}orie des alg\`{e}bres non-associatives,
l'alg\`{e}bre de Lie des d\'{e}rivations d'une alg\`{e}bre $\E$ est un outil
tr\`{e}s important pour \'{e}tudier la structure de l'alg\`{e}bre.

Dans \cite[Theorem~4.1]{Farrukh2014}, les Auteurs trouvent la conditions sous
laquelle une alg\`{e}bre g\'{e}n\'{e}tique de dimension $2$ s'identifie \`{a}
une alg\`{e}bre d'\'{e}volution puis, ils prouvent l'existence d'une
d\'{e}rivation non-triviale dans les alg\`{e}bres g\'{e}n\'etiques en
dimension $2$ \cite[Corrolary~5.2]{Farrukh2014}.

Dans \cite{Qaralleh2017}, les Auteurs d\'{e}crivent les d\'{e}rivations dans les alg\`{e}bres
d'\'{e}volution nilpotentes et resolubles en dimension $3$. Ils d\'{e}crivent
aussi les d\'{e}rivations dans les sous-alg\`{e}bres d'\'{e}volution
complexes de dimension $2$.

Dans \cite{Camacho2013}, les Auteurs \'{e}tudient les d\'{e}rivations des alg\`{e}bres d'\'{e}volution de
dimension $n$, lorsque le rang de la matrice des constantes de structures est
$n - 1$ ou $n$.

Dans la section $2$, nous montrons qu'il est suffisant de
caract\`{e}riser les d\'{e}rivations dans les nil-alg\`{e}bres
d'\'{e}volution \`{a} puissances associatives puis nous caract\'{e}risons les
d\'{e}rivations dans les nil-alg\`{e}bres d'\'{e}volution d\'{e}composable
\`{a} puissances associatives.

Dans la section $3$, nous caract\'erisons les d\'erivations dans les nil-alg\`ebres d'\'evolution d\'ecomposables \`a puissances associatives.

Dans la section $4$, nous  \'{e}tudions les d\'{e}rivations dans les nil-alg\`{e}bres d'\'{e}volution ind\'{e}composable
et associatives.

Dans la section $5$, nous d\'{e}crivons les d\'{e}rivations
dans les nil-alg\`{e}bres d'\'{e}volution ind\'{e}composable \`{a} puissances
associatives et qui ne sont pas associatives.




\section{Pr\'{e}liminaires}

Soient $K$ un corps commutatif et $\E$ une $K$-alg\`{e}bre commutative. On d\'{e}finit par $R_{a} : \E \longrightarrow \E$,
$x \longmapsto xa$ \textit{la multiplication \`{a} droite} par l'\'{e}l\'{e}ment $a$ de $\E$. On note $R(\E)$ l'ensemble des multiplications à droite de $\E$. Soit $\Le(\E)$ l'alg\`ebre de Lie des transformations de $\E$.

\begin{defi} Une d\'erivation $d$ de $\E$ est dite \textit{int\'erieure} si $d\in \Le(\E)$. On note $\In(\E)=\D(\E)\cap\Le(\E)$, l'ensemble des d\'erivations int\'erieures de $\E$. C'est un id\'eal de l'alg\`ebre des d\'erivations $\D(\E)$.
\end{defi}
Dans les cas o\`u l'alg\`ebre $\E$ est associative (resp. de Jordan) $\Le(\E)=R(\E)$ (resp. $\Le(\E)=R(\E)+[R(\E),R(\E)]$) \cite{schafer}.

\textit{Les puissances principales} d'un élément
$a \in \E$ sont définies par $a^1=a$ et $a^{k+1}=a^ka$ tandis que celles de $\E$ sont d\'{e}finies par $\E^{1} = \E,\quad \E^{k + 1} = \E^{k}\E$ ($k\geq 1$).

\begin{defi}
On dira que l'algèbre $\E$ est :

$i)$ \textit{nilpotente}  s'il existe un entier non nul $n$ tel que $\E^{n} = 0$, un tel plus petit entier est appel\'{e} \textit{l'indice de nilpotence} ;

$ii)$ \textit{nil}, s'il existe un entier non nul $n(a)$ tel que $a^{n(a)} = 0,$ pour tout $a \in \E$, un tel plus petit entier est appel\'{e} \textit{le nilindice} de $\E$.
\end{defi}

\begin{defi} L'algèbre $\E$ est dite :

$i)$ \textit{associative} si pour tous $x,$ $y,$ $z \in \E,$ $(x, y, z) = 0$ o\`{u} $(x, y, z) = (xy)z - x(yz)$ d\'{e}signe l'associateur des \'{e}l\'{e}ments $x,$ $y,$ $z$ de $\E$ ;

$ii)$ \textit{de Jordan} si elle vérifie $(x^{2}, y, x) = 0$, pour tous $x, y \in \E$ ;

$iii)$ \textit{\`{a} puissances associatives} si pour tout $x \in \E$, la sous alg\`{e}bre engendr\'{e}e par $x$ est associative. Autrement dit, pour tout $x \in \E$, $x^{i}x^{j} = x^{i + j}$ pour tous entiers $i, j\geq 1.$
\end{defi}

\begin{defi}
L'algèbre $\E$ est dite \textit{\`{a} quatri\`{e}me puissances associatives} si $x^{2}x^{2} = x^{4}$ pour tout $x \in \E$.
\end{defi}

\begin{theo}[\cite{Albert1}] En caract\'{e}ristique $\neq 2, 3, 5$, l'algèbre $\E$ est \`{a} puissances associatives si et seulement si $x^{2}x^{2} = x^{4}$, pour tout $x \in A.$
\end{theo}

\begin{defi}\label{indecom} L'algèbre $\E$ est \textit{décomposable} s'il existe des idéaux non nuls $\I$ et $\J$ tels que $\E=\I\oplus\J$. Dans le cas contraire elle est \textit{indécomposable}.
\end{defi}

\begin{defi}
Si l'algèbre $\E$ admet une base $B = \{e_{1}, e_{2}, \ldots, e_{n}\}$ telle que
\begin{eqnarray}\label{Eq1}
e_{i}e_{j} = 0 \textrm{ pour tous $1 \leq i \neq j \leq n$ et }e_{i}^{2} = \sum_{k=1}^{n} a_{ik}e_{k}\textrm{ pour tout $1 \leq i \leq n$},
\end{eqnarray}
alors on dit que $\E$ est une \textit{$K$-alg\`{e}bre d'\'{e}volution} et $B$ est \textit{une base naturelle} de $\E$. La matrice $M = (a_{ik})_{1 \leq i, k \leq n}$
est \textit{la matrice des constantes de structures} de $\E$ relativement \`{a} la base naturelle~$B.$
\end{defi}

\begin{defi} Une alg\`{e}bre $\E$ de dimension finie $n + 1$ est dite \textit{dimensionnellement nilpotente (ADN)} s'il existe une d\'{e}rivation $d$ de $\E$ telle que $d^{n} = 0$.
\end{defi}

Dans ce cas, il existe une base $\{e_{0}, e_{1}, \ldots, e_{n}\}$ de $\E$ telle que $d(e_{i}) = e_{i + 1}$, $d(e_{n}) = 0$ avec $i = 0, \ldots, n - 1$. Une telle base est
dite \textit{adapt\'{e}e}.

\begin{theo} La seule alg\`{e}bre d'\'{e}volution ind\'{e}composable qui est ADN est $N_{2, 2} : e_{0}^{2} = e_{1}, e_{1}^{2} =~0$.
\end{theo}

\begin{proof} On suppose que $\E$ est une ADN d'\'{e}volution de dimension finie $n + 1$ dans la base adapt\'{e}e $\{e_{0}, e_{1}, \ldots, e_{n}\}$.
En d\'{e}rivant $e_{i}^{2}$, on obtient $d(e_{i}^{2}) = 2e_{i}e_{i + 1} = 0$, donc $e_{i}^{2} \in\ker(d)=Ke_n$. Il existe alors des scalaires
$\al_{in}$, tels que $e_{i}^{2} = \alpha_{in}e_{n}$, pour $i = 0, \ldots, n - 1$, et $e_{n}^{2} = 0$. On en d\'{e}duit que $\E$ est associative
(\cite[\S~4.1]{Ouatt2018}). En d\'{e}rivant $0 = e_{i}e_{j}$ pour $0 \leq i \neq j \leq n - 1$, on
obtient $0 = d(e_{i})e_{j} + e_{i}d(e_{j}) = e_{i + 1}e_{j} + e_{i}e_{j + 1}$ et en prenant $i + 1 = j$, l'\'{e}galit\'{e} devient
 $0 = e_{j}^{2} + e_{j - 1}e_{j + 1} = e_{j}^{2}$ pour $j = 1, \ldots, n - 1$. On d\'{e}duit de tout ce qui pr\'{e}c\`{e}de que
 $e_{0}^{2} = \alpha_{0n}e_{n}$, $e_{i}^{2} = 0$ pour $i = 1, \ldots, n.$ Donc $\E = N_{n + 1, 1}$ ou $\E = N_{n - 1, 1}\oplus N_{2, 2}$ selon que
 $\al_{0n}=0$, o\`{u} $N_{j, 1}$ est une z\'{e}ro alg\`{e}bre de dimension $j$ et $N_{2, 2} : e_{0}^{2} = e_{1}$, $e_{1}^{2} = 0$ est
 une alg\`{e}bre d'\'{e}volution, d'o\`{u} le th\'{e}or\`{e}me.
\end{proof}

Soient $\E$ une alg\`{e}bre d'\'{e}volution  dont la table de multiplication, dans une base naturelle $B = \{e_{1}, e_{2}, \ldots, e_{n}\}$, est donn\'{e}e par \eqref{Eq1}.

On d\'{e}finit  \textit{l'annulateur} de $\E$ par $\ann(\E) = \{x \in \E : x\E = 0\}$ et dans \cite[Lemme 2.7]{Elduque2015}, les Auteurs montrent
que $\ann(\E) = span\{e_{i} \in B \mid e_{i}^{2} = 0\}.$

Soit $d$ une d\'{e}rivation dans $\E$ ; on pose $d(e_{i}) = \sum_{j = 1}^{n}d_{ij}e_{j}$ pour tout $1 \leq i \leq n$. Dans \cite{Tian2008}, l'Auteur montre que
\[
\D(\E) = \{d \in End(\E) ; a_{jk}d_{ji} + a_{ik}d_{ij} = 0 \textrm{ pour } i \neq j \textrm{ et pour tout } k, 2d_{ii}a_{ik} = \sum_{j = 1}^{n}a_{ij}d_{kj}\}.
\]

On suppose que $\E$ est une nil-alg\`{e}bre d'\'{e}volution telle que $\dim(ann(\E)) = n - s$ o\`{u} $s$ est un entier non nul. Sans perdre la g\'{e}n\'{e}ralit\'{e}, on peut poser $ann(\E) = <e_{s + 1}, \ldots, e_{n}>$.
Pour $1 \leq i \neq j \leq s$, en d\'{e}rivant $0 = e_{i}e_{j}$, on obtient $0 = d(e_{i}e_{j}) = e_{i}d(e_{j}) + d(e_{i})e_{j} = d_{ij}e_{i}^{2} + d_{ji}e_{j}^{2}$. On distingue alors deux cas :
\begin{itemize}
  \item[--] Premier cas : Les vecteurs $e_{i}^{2}$ et $e_{j}^{2}$ sont lin\'eairement indépendants, alors $d_{ij} = d_{ji} = 0$. \hfill $(a)$
  \item[--] Deuxi\`{e}me cas : $e_{i}^{2} = \beta_{ij}e_{j}^{2}$ avec $\beta_{ij} \in K^{*}$ (car $e_{i}^{2} \neq 0$ et $e_{j}^{2} \neq 0$), alors $d_{ji} = -\beta_{ij}d_{ij}$.\hfill $(b)$
\end{itemize}
Dans la suite, on suppose que le corps $K$ est de caract\'{e}ristique $\neq 2$ et $\E$ est une alg\`ebre d'\'evolution, \`{a} puissances associatives.

\begin{theo}[\cite{Ouatt2018}, Theorem 3.5.4]\label{Wedder}
Si $\E$ est non nil, alors $\E$ admet $s$ idempotents deux \`{a} deux orthogonaux $u_{1}, u_{2}, \ldots, u_{s}$, non nuls, tels~que
\begin{equation}\label{Eq4}
\E = Ku_{1}\oplus Ku_{2}\oplus \cdots \oplus Ku_{s} \oplus N,
\end{equation}
somme directe d'alg\`{e}bres, o\`{u} $s \geq 1$ est un entier et $N$ est soit nul, soit une nil-alg\'{e}bre d'\'{e}volution \`{a} puissances associatives.

De plus $\E_{ss} = Ku_{1}\oplus Ku_{2}\oplus \cdots \oplus Ku_{s}$ est la composante semi-simple de $\E$ et $N = Rad(\E)$ est le nil radical de $\E$.
\end{theo}

\begin{prop} Si $\E$ est non nil, alors $\D(\E) = \D(N)$ o\`{u} $N$ est le nil radical dans la d\'ecomposition de Wedderburn de $\E$.
\end{prop}

\begin{proof} Soient $\E = Ku_{1}\oplus Ku_{2}\oplus \cdots \oplus Ku_{s} \oplus N$ la d\'{e}composition de Wedderburn de $\E$, $d \in \D(\E)$ et $v = \sum_{i = 1}^{s}\alpha_{i}u_{i}$. Pour tout $1 \leq i \leq s$, on a $d(u_{i}) = d(u_{i}^{2}) = 2u_{i}d(u_{i}) = 2d_{ii}u_{i}^{2} = 2d_{ii}u_{i}$ entra\^{\i}ne $2d_{ii} = d_{ii},$ donc $d_{ii} = 0$ et $d(u_{i}) = 0$. Ainsi $d(v) = \sum_{i = 1}^{s}\alpha_{i}d(u_{i}) = 0,$ d'o\`{u} le r\'{e}sultat.
\end{proof}

On en d\'{e}duit que le calcul des d\'{e}rivations dans les alg\`{e}bres d'\'{e}volution, \`{a}  puissances associatives, se ram\`{e}ne \`{a} celui des d\'erivations des nil-alg\`{e}bres d'\'{e}volution \`{a}  puissances associatives.

\begin{exe}\cite[Theorem 4.1.1]{Ouatt2018}\label{exe1} Soit $N_{1, 1} : e_{1}^{2} = 0$, l'unique nil-alg\`{e}bre d'\'{e}volution associative ind\'{e}composable de dimension $1$. On a $\D(N_{1, 1}) \simeq K$ alg\`ebre de Lie ab\'elienne.
\end{exe}

\begin{exe}\cite[Theorem 4.1.2]{Ouatt2018}\label{exe2} Soient $N_{2, 2} : e_{1}^{2} = e_{2}, e_{2}^{2} = 0$, l'unique nil-alg\`{e}bre d'\'{e}volution associative
 ind\'{e}composable de dimension $2$ et $d \in \D(N_{2, 2}).$ On a $d(e_{2}) = d(e_{1}^{2}) = 2e_{1}d(e_{1}) = 2d_{11}e_{1}^{2} = 2d_{11}e_{2}$. Ainsi
 $d(e_{1}) = d_{11}e_{1} + d_{21}e_{2}$ et $d(e_{2}) = 2d_{11}e_{2}$. On en d\'{e}duit que $\D(N_{2, 2}) \simeq K^{2}.$

Soient $d, d'$ des d\'{e}rivations de $N_{2, 2}$. On a
$d\circ d' (e_{1}) =  d_{11}'d(e_{1}) + d_{21}'d(e_{2}) = d_{11}'d_{11}e_{1} + (d_{11}'d_{21} + 2d_{21}'d_{11})e_{2}$ et $[d, d'](e_{1}) = (d_{21}'d_{11} - d_{21}d_{11}')e_{2}$.
Aussi $d\circ d' (e_{2}) =  2d_{11}'d(e_{2}) = 4d_{11}'d_{11}e_{2}$ et $[d, d'](e_{2}) = 0$. On en d\'{e}duit que $\D(N_{2, 2})$ est l'alg\`ebre de Lie non-ab\'elienne.
\end{exe}

\begin{exe}\cite[Table~1]{Ouatt2018}\label{exe3} Soient $N_{3, 2} = N_{2, 2}\oplus N_{1, 1} : e_{1}^{2} = e_{2}, e_{2}^{2} = e_{3}^{2} = 0$ et $d \in \D(N_{3, 2}).$ On a $d(e_{2}) = d(e_{1}^{2}) = 2e_{1}d(e_{1}) = 2d_{11}e_{1}^{2} = 2d_{11}e_{2}$ ;
$0 = d(e_{1}e_{j}) = e_{1}d(e_{j}) + d(e_{1})e_{j} = d_{1j}e_{1}^{2} + d_{j1}e_{j}^{2} = d_{1j}e_{2}$ ($j = 2, 3$), donc $d_{12} = d_{13} = 0.$
Ainsi $d(e_{1}) = d_{11}e_{1} + d_{21}e_{2} + d_{31}e_{3}$, $d(e_{2}) = 2d_{11}e_{2}$ et $d(e_{3}) = d_{23}e_{2} + d_{33}e_{3}.$ On en d\'{e}duit que $\D(N_{3, 2}) \simeq K^{5}.$
\end{exe}

\begin{exe}\cite[Table~1]{Ouatt2018}\label{exe4} Soient $N_{4, 4} = N_{2, 2}\oplus N_{2, 2} : e_{1}^{2} = e_{2}, e_{2}^{2} = 0, e_{3}^{2} = e_{4}, e_{4}^{2} = 0$ et $d \in \D(N_{4, 4}).$ On a $d(e_{2}) = d(e_{1}^{2}) = 2e_{1}d(e_{1}) = 2d_{11}e_{1}^{2} = 2d_{11}e_{2}$ ; $d(e_{4}) = d(e_{3}^{2}) = 2e_{3}d(e_{3}) = 2d_{33}e_{3}^{2} = 2d_{33}e_{4}$. On a aussi $0 = d(e_{1}e_{j}) = e_{1}d(e_{j}) + d(e_{1})e_{j} = d_{1j}e_{1}^{2} + d_{j1}e_{j}^{2} = d_{1j}e_{2} + d_{j1}e_{j}^{2}$ ($j = 2, 3, 4$), donc $d_{12} = 0$, $d_{13} = d_{31} = 0$ et $d_{14} = 0$ en prenant respectivement $j = 2$, $j = 3$ et $j = 4$ ; $0 = d(e_{3}e_{j}) = e_{3}d(e_{j}) + d(e_{3})e_{j} = d_{3j}e_{3}^{2} + d_{j3}e_{j}^{2} = d_{3j}e_{4} + d_{j3}e_{j}^{2}$ ($j = 2, 4$), donc $d_{32} = 0$ et $d_{34} = 0$ en prenant respectivement $j = 2$ et $j = 4$. Ainsi $d(e_{1}) = d_{11}e_{1} + d_{21}e_{2} + d_{41}e_{4}$, $d(e_{2}) = 2d_{11}e_{2}$, $d(e_{3}) = d_{23}e_{2} + d_{33}e_{3} + d_{43}e_{4}$, $d(e_{4}) = 2d_{33}e_{4}$. On en d\'{e}duit que $\D(N_{4, 4}) \simeq K^{6}.$
\end{exe}

\section{Nil-alg\`{e}bres d\'{e}composables}

Soit $\E = I_{1} + I_{2}$ une nil-alg\`{e}bre d'\'{e}volution d\'{e}composable.

\begin{theo}[Caract\'{e}risation des d\'{e}rivations] Soit $d : \E \longrightarrow \E$ une d\'{e}rivation. Alors $d$ est d\'{e}termin\'{e}e par un unique
quadruplet $(f_{d}, g_{d}, \ell_d, k_{d})$ v\'{e}rifiant les conditions suivantes~:
\begin{compactenum}[$i)$]
\item $d(x_{1}) = f_{d}(x_{1}) + \ell_d(x_{1})$ avec $x_{1} \in I_{1}$ ;
\item $d(x_{2}) = g_{d}(x_{2}) + k_{d}(x_{2})$ avec $x_{2} \in I_{2}$ ;

Les applications lin\'eaires $f_{d} \in End(I_{1})$, $g_{d} \in End(I_{2})$, $\ell_d \in \Hom (I_{1}, I_{2})$ et $k_{d} \in \Hom (I_{2}, I_{1})$ v\'{e}rifient~:
\item $f_{d} \in \D(I_{1})$ ;
\item $g_{d} \in \D(I_{2})$ ;
\item $\ell_d \in \Hom ^{0}(I_{1}, I_{2}) = \{h \in \Hom (I_{1}, I_{2}) \mid h(I_{1}^{2}) = 0$ et $h(I_{1}) \subseteq ann(I_{2})\}$ ;
\item $k_{d} \in \Hom ^{0}(I_{2}, I_{1}) = \{h \in \Hom (I_{2}, I_{1}) \mid h(I_{2}^{2}) = 0$ et $h(I_{2}) \subseteq ann(I_{1})\}$.
\end{compactenum}
\end{theo}

\begin{proof} Soit $d$ une d\'{e}rivation de $\E = I_{1} + I_{2}$. Puisque $d(x_{i}) \in \E$ (avec $x_{i} \in I_{i}$), alors posons
$d(x_{1}) = f_{d}(x_{1}) + \ell_d(x_{1})$ et $d(x_{2}) = k_{d}(x_{2}) + g_{d}(x_{2})$ avec $f_{d}(x_{1}), k_{d}(x_{2}) \in I_{1}$ et $\ell_d(x_{1}), g_{d}(x_{2}) \in I_{2}$.
On v\'{e}rifie que $f_{d} \in End(I_{1})$, $g_{d} \in End(I_{2})$, $\ell_d \in \Hom (I_{1}, I_{2})$ et $k_{d} \in \Hom (I_{2}, I_{1})$, d'o\`u $i)$ et $ii)$.

Soient $x_{1}, x_{1}' \in I_{1}$ et $x_{2}, x_{2}' \in I_{2}$.
On a~: $d(x_{1}x_{1}') = x_{1}(f_{d}(x_{1}') + \ell_d(x_{1}')) + (f_{d}(x_{1}) + \ell_d(x_{1}))x_{1}' = x_{1}f_{d}(x_{1}') + f_{d}(x_{1})x_{1}' \in I_{1}$ ;
puisque $d(x_{1}x_{1}') = f_{d}(x_{1}x_{1}') + \ell_d(x_{1}x_{1}')$ alors $f_{d}(x_{1}x_{1}') = x_{1}f_{d}(x_{1}') + f_{d}(x_{1})x_{1}'$ et $\ell_d(x_{1}x_{1}') = 0$,
donc $f_{d} \in \D(I_{1})$, on obtient $iii)$ et $\ell_d(I_{1}^{2}) = 0.$

En calculant $d(x_2x'_2)$, on montre de m\^{e}me que $g_{d} \in \D(I_{2})$ d'o\`{u} $iv)$ et $k_{d}(I_{2}^{2}) = 0$.
On a~: $0 = d(x_{1}x_{2}) = x_{1}(g_{d}(x_{2}) + k_{d}(x_{2})) + (f_{d}(x_{1}) + \ell_d(x_{1}))x_{2} = x_{1}k_{d}(x_{2}) + \ell_d(x_{1})x_{2}$ entra\^{\i}ne $x_{1}k_{d}(x_{2}) = 0$ et $\ell_d(x_{1})x_{2} = 0$, donc $\ell_d(I_{1}) \subseteq ann(I_{2})$ et $k_{d}(I_{2}) \subseteq ann(I_{1})$. Comme $\ell_d(I_{1}^{2}) = 0$ et $\ell_d(I_{1}) \subseteq ann(I_{2})$ alors, nous obtenons $v)$ ; aussi $k_{d}(I_{2}^{2}) = 0$ et $k_{d}(I_{2}) \subseteq ann(I_{1})$ entra\^{\i}nent $vi)$.
\end{proof}

Soient $d$ et $d'$ deux d\'{e}rivations de $\E.$ Comme $[d, d']$ est une d\'{e}rivation, alors nous d\'{e}terminons l'unique quadruplet qui lui est associ\'{e}.

\begin{theo}\label{Mult} Soient $d$ et $d'$ deux d\'{e}rivations de $\E.$ Alors l'unique quadruplet associ\'{e} \`{a} $[d, d']$ est $(f_{[d, d']}, g_{[d, d']}, \ell_{[d, d']}, k_{[d, d']})$ v\'{e}rifiant~:
\begin{compactenum}[$i)$]
\item $f_{[d, d']} = [f_{d}, f_{d'}] + (k_{d}\circ\ell_{d'} - k_{d'}\circ\ell_d)$ ;
\item $\ell_{[d, d']} = (\ell_d\circ f_{d'} - \ell_{d'}\circ f_{d}) + (g_{d}\circ\ell_{d'} - g_{d'}\circ\ell_d)$ ;
\item $g_{[d, d']} = [g_{d}, g_{d'}] + (\ell_d\circ k_{d'} - \ell_{d'}\circ k_{d})$ ;
\item $k_{[d, d']} = (k_{d}og_{d'} - k_{d'}og_{d}) + (f_{d}\circ k_{d'} - f_{d'}\circ k_{d})$.
\end{compactenum}
\end{theo}

\begin{proof} On a $d\circ d' (x_{1}) = d(f_{d'}(x_{1})) + d(\ell_{d'}(x_{1})) = f_{d}\circ f_{d'}(x_{1}) + \ell_d\circ f_{d'}(x_{1}) + g_{d}\circ\ell_{d'}(x_{1}) + k_{d}\circ\ell_{d'}(x_{1})$ et $[d, d'](x_{1}) = (f_{d}\circ f_{d'} - f_{d'}\circ f_{d})(x_{1}) + (\ell_d\circ f_{d'} - \ell_{d'}\circ f_{d})(x_{1}) + (g_{d}\circ\ell_{d'} - g_{d'}\circ\ell_d)(x_{1}) + (k_{d}\circ\ell_{d'} - k_{d'}\circ\ell_d)(x_{1}).$ Donc $f_{[d, d']} = [f_{d}, f_{d'}] + (k_{d}\circ\ell_{d'} - k_{d'}\circ\ell_d)$ et $l_{[d, d']} = (\ell_d\circ f_{d'} - \ell_{d'}\circ f_{d}) + (g_{d}\circ\ell_{d'} - g_{d'}\circ\ell_d)$, d'o\`{u} $i)$ et $ii)$. On montre de m\'{e}me $iii)$ et $iv)$ en calculant $[d, d'](x_{2})$.
\end{proof}

Notons $L_{K}(\E) = \D(I_{1})\times \D(I_{2})\times \Hom ^{0}(I_{1}, I_{2})\times \Hom ^{0}(I_{2}, I_{1})$ ; la multiplication dans $L_{K}(\E)$ est d\'{e}finie par $[(f_{d}, g_{d}, \ell_d, k_{d}), (f_{d'}, g_{d'}, \ell_{d'}, k_{d'})] = (f_{[d, d']}, g_{[d, d']}, l_{[d, d']}, k_{[d, d']})$ o\`{u} le quadruplet ($f_{[d, d']}$, $g_{[d, d']}$, $\ell_{[d, d']}$, $k_{[d, d']}$) satisfait les relations $i)$, $ii)$, $iii)$ et $iv)$ du Th\'{e}or\`{e}me~\ref{Mult}.

\begin{prop}\label{Iso}  L'application $\varphi : \D(\E) \longrightarrow L_{K}(\E),$ $d \longmapsto (f_{d}, g_{d}, \ell_d, k_{d})$ est un isomorphisme d'alg\`{e}bres de Lie.
\end{prop}

\begin{proof} On a
$[\varphi(d), \varphi(d')] = [(f_{d}, g_{d}, \ell_d, k_{d}), (f_{d'}, g_{d'}, \ell_{d'}, k_{d'})] = (f_{[d, d']}, g_{[d, d']}, \ell_{[d, d']}, k_{[d, d']}) \\ = \varphi([d, d'])$ et $ker(\varphi) = \{0\}$, donc $\varphi$ est un monomorphisme d'alg\`{e}bres de Lie.

Soient $(f, g, \ell, k) \in L_{K}(\E)$ et $d : \E \longrightarrow \E$ un endomorphisme d\'{e}fini par $d(x_{1}) = f(x_{1}) + \ell(x_{1})$ et $d(x_{2}) = g(x_{2}) + k(x_{2})$. On a $d(x_{1}x_{1}') = f(x_{1}x_{1}') = x_{1}f(x_{1}') + f(x_{1})x_{1}' = x_{1}(f(x_{1}') + \ell(x_{1}')) + (f(x_{1}) + \ell(x_{1}))x_{1}' = x_{1}d(x_{1}') + d(x_{1})x_{1}'$. On montre de m\^{e}me que $d(x_{2}x_{2}') = x_{2}d(x_{2}') + d(x_{2})x_{2}'$.

Aussi $d(x_{1}x_{2}) =  0 = x_{1}(g(x_{2}) + k(x_{2})) + (f(x_{1}) + \ell(x_{1}))x_{2} = x_{1}d(x_{2}) + d(x_{1})x_{2}$ car $x_{1}g(x_{2}) = x_{1}k(x_{2}) = x_{2}f(x_{1}) = x_{2}l(x_{1}) = 0$. Par cons\'{e}quent $\varphi$ est surjectif, d'o\`{u} le th\'{e}or\`{e}me.
\end{proof}

\begin{exe}\label{exe5} On consid\`{e}re les alg\`{e}bres $N_{1, 1}$, $N_{2, 2}$, $N_{3, 2}$ et $N_{4, 4}$ d\'{e}finies ci-dessus. On a

$\Hom ^{0}(N_{1, 1}, N_{2, 2}) = \{h \in \Hom (N_{1, 1}, N_{2, 2}) ; h(N_{1, 1}^{2}) = 0$ et $h(N_{1, 1}) \subseteq ann(N_{2, 2}) \} \simeq K.$

$\Hom ^{0}(N_{2, 2}, N_{1, 1}) = \{h \in \Hom (N_{2, 2}, N_{1, 1}) ; h(N_{2, 2}^{2}) = 0$ et $h(N_{2, 2}) \subseteq ann(N_{1, 1}) \} \simeq K$, car $N_{2, 2}^{2} = <e_{2}>$, $0 = h(e_{1}^{2}) = h(e_{2})$ et $h(e_{1}) =  \alpha e_{3}.$

$\Hom ^{0}(N_{2, 2}, N_{2, 2}) = \{h \in \Hom (N_{2, 2}, N_{2, 2}) ; h(N_{2, 2}^{2}) = 0$ et $h(N_{2, 2}) \subseteq ann(N_{2, 2}) \} \simeq K$
Car $N_{2, 2}^{2} = <e_{2}>$, $0 = h(e_{1}^{2}) = h(e_{2})$ et $h(e_{1}) =  \alpha e_{4}.$

On en d\'{e}duit que $\D(N_{2, 2}\oplus N_{1, 1}) \simeq K^{2}\times K\times K\times K$ et $\D(N_{2, 2}\oplus N_{2, 2}) \simeq K^{2}\times K^{2}\times K\times K.$
\end{exe}

De tout ce qui pr\'{e}c\`{e}de, on en d\'{e}duit qu'il est suffisant de d\'{e}terminer les d\'{e}rivations dans les nil-alg\`{e}bres d'\'{e}volution  ind\'{e}composables.

\begin{lem}[\cite{Elduque2016}, Corollary~2.6] Soit $\E$ une nil-algèbre d'évolution de dimension finie telle que $\dim_K(\ann(\E))\geq \frac12\dim_K(\E)\geq 1$. Alors $\E$ est décomposable.
\end{lem}

\begin{prop}[\cite{Ouatt2018}, Table~$1$] Soit $\E$ une nil-alg\`ebre d'\'evolution ind\'{e}composable \`a puissances associatives de dimension $\leq 4$. Alors $\E$ est isomorphe \`a une et une seule des alg\`ebres dans la Table~$1$.
{\footnotesize
\begin{longtable}[c]{|p{6mm}|c|p{6,6cm}|c|c|}
 \caption*{\underline{\textbf{Table 1:}} $\dim(\E)\leq 4$}\\
  \hline
$\dim$&  {$\E$} &   {\small Multiplication} &   {\small dim(ann$(\E)$)} &   {\small Associative}  \\\hline
\endhead \hline
\endfoot
\multirow{1}{6mm}{$\mathbf{1}$}&  {$N_{1,1}$} &   {$e_{1}^{2} = 0$} &   {$1$} &   {\small Oui}  \\
 \hline
\multirow{1}{6mm}{$\mathbf{2}$} &{$N_{2,2}$} &   {$e_{1}^{2} = e_{2},$ $e_{2}^{2} = 0$} &   {$1$} & {\small Oui} \\
 \hline
\multirow{1}{6mm}{$\mathbf{3}$}&{$N_{3,3}(\alpha)$} &   {$e_{1}^{2} = e_{3},$ $e_{2}^{2} = \alpha e_{3},$ $e_{3}^{2} = 0$ avec $\alpha \in K^{*}$} &   {$1$} &   {\small Oui} \\
 \hline
\multirow{2}{6mm}{$\mathbf{4}$}&{$N_{4,5}(\alpha, \beta)$} & {$e_{1}^{2} = e_{4},$ $e_{2}^{2} = \alpha e_{4},$ $e_{3}^{2} = \beta e_{4},$ $e_{4}^{2} = 0$ avec $\alpha, \beta \in K^{*}$} &   {$1$} &   {\small Oui}  \\\cline{2-5}
  &{$N_{4,6}$} & {$e_{1}^{2} = e_{2} + e_{3},$ $e_{2}^{2} = e_{4},$ $e_{3}^{2} = -e_{4},$ $e_{4}^{2} = 0$} &   {$1$} &   {\small Non}  \\
\hline
\end{longtable}
}
\end{prop}

\begin{prop}[\cite{Ouatt2018}, Table~$2$]
Soit $\E$ une nil-alg\`ebre d'\'evolution ind\'{e}composable \`a puissances associatives de dimension $5$. Alors $\E$ est isomorphe \`a une et une seule des alg\`ebres dans la Table~$2$.

{\footnotesize
\begin{longtable}[c]{|c|p{7cm}|c|c|}
 \caption*{\underline{\textbf{Table 2:}} $\dim(\E)=5$}\\
  \hline
  {$\E$} &   {\small Multiplication} &   {\small dim(ann($\E$))} &   {\small Associative}  \\\hline
\endhead \hline
\endfoot
  {$N_{5,8}(\alpha, \beta, \gamma)$} &   {$e_{1}^{2} = e_{5},$ $e_{2}^{2} = \alpha e_{5},$ $e_{3}^{2} = \beta e_{5},$ $e_{4}^{2} = \gamma e_{5},$ $e_{5}^{2} = 0$ avec $\alpha, \beta, \gamma \in K^{*}$} &   {$1$} &   {\small Oui}  \\\hline
  {$N_{5,9}(\alpha, \beta)$} &   {$e_{1}^{2} = e_{4},$ $e_{2}^{2} = \alpha e_{4} + \beta e_{5},$ $e_{3}^{2} = e_{5},$ $e_{4}^{2} = e_{5}^{2} = 0$ avec $\alpha, \beta \in K^{*}$} &   {$2$} &   {\small Oui}  \\\hline
  {$N_{5,10}(\alpha)$} &   {$e_{1}^{2} = e_{2} + e_{3},$ $e_{2}^{2} = e_{5},$ $e_{3}^{2} = -e_{5},$ $e_{4}^{2} = \alpha e_{5} ,$ $e_{5}^{2} = 0$ avec $\alpha \in K^{*}$} &   {$1$} &   {\small Non}  \\\hline
  {$N_{5,11}(\alpha)$} &   {$e_{1}^{2} = e_{2} + e_{3},$ $e_{2}^{2} = e_{5},$ $e_{3}^{2} = -e_{5},$ $e_{4}^{2} = \alpha(e_{2} + e_{3}),$ $e_{5}^{2} = 0$ avec $\alpha \in K^{*}$} &   {$1$} &   {\small Non}  \\\hline
  {$N_{5,12}(\alpha, \beta)$} &   {$e_{1}^{2} = e_{2} + e_{3},$ $e_{2}^{2} = e_{5},$ $e_{3}^{2} = -e_{5},$ $e_{4}^{2} = \alpha(e_{2} + e_{3}) +  \beta e_{5},$ $e_{5}^{2} = 0$ avec $\alpha, \beta \in K^{*}$} &   {$1$} &   {\small Non}  \\
\hline
\end{longtable}
}
\end{prop}

\begin{prop}[\cite{Ouatt2018}, Table~$3$]
Soit $\E$ une nil-alg\`ebre d'\'evolution ind\'{e}composable, \`a puissances associatives, de dimension $6$. Alors $\E$ est isomorphe \`a une et une seule des alg\`ebres dans la Table~$3$.
 {\footnotesize
 \begin{longtable}[c]{|c|p{7cm}|c|c|}
 \caption*{\underline{\textbf{Table 3:}} $\dim(\E)=6$}\\
  \hline
    {$\E$} &   {\small Multiplication} &   {\small dim(ann($\E$))} & {\footnotesize Associative}  \\\hline
\endhead \hline
\endfoot{$N_{6,16}(\alpha, \beta, \gamma, \delta)$} &   {$e_{1}^{2} = e_{6}, e_{2}^{2} = \alpha e_{6},$ $e_{3}^{2} = \beta e_{6},$ $e_{4}^{2} = \gamma e_{6},$ $e_{5}^{2} = \delta e_{6}$, $e_{6}^{2} = 0$ avec $\alpha, \beta, \gamma, \delta \in K^{*}$} &   {$1$} &   {\small Oui}  \\\hline
{$N_{6,17}(\alpha, \beta, \gamma)$} &   {$e_{1}^{2} = e_{5}$, $e_{2}^{2} = \alpha e_{5} +  \beta e_{6}$, $e_{3}^{2} =  \gamma e_{6}$, $e_{4}^{2} = e_{6}$, $e_{5}^{2} = e_{6}^{2} = 0$ avec $\alpha\be\gamma\neq 0$}  &   {$2$} & {\small Oui}  \\\hline
{$N_{6,18}(\alpha, \beta, \gamma,\delta)$} &   {$e_{1}^{2} = e_{5}$, $e_{2}^{2} = \al e_{5} +  \be e_{6}$, $e_{3}^{2} = \gamma e_{5}+\de e_6$, $e_{4}^{2} = e_{6}$, $e_{5}^{2} = e_{6}^{2} = 0$ avec $\al\be\ne0$, $\g\de\ne 0$ et $\alpha\delta-\beta\gamma \neq 0$} &   {$2$} &   {\small Oui}  \\\hline
{$N_{6,19}(\alpha, \beta)$} &   {$e_{1}^{2} = e_{2} + e_{3},$ $e_{2}^{2} = e_{6},$ $e_{3}^{2} = -e_{6},${} $e_{4}^{2} = \alpha e_{6},$  $e_{5}^{2} = \beta e_{6},$ $e_{6}^{2} = 0$ avec $\alpha, \beta \in K^{*}$} &   {$1$} & {\small Non} \\\hline
  {$N_{6,20}(\alpha, \beta)$} &   {$e_{1}^{2} = e_{2} + e_{3},$ $e_{2}^{2} = e_{6},$ $e_{3}^{2} = -e_{6},${} $e_{4}^{2} = \alpha(e_{2} + e_{3}),$  $e_{5}^{2} = \beta e_{6},$ $e_{6}^{2} = 0$ avec $\alpha, \beta \in K^{*}$} &   {$1$} & {\small Non}  \\\hline
  {$N_{6,21}(\alpha, \beta, \gamma)$} &   {$e_{1}^{2} = e_{2} + e_{3},$ $e_{2}^{2} = e_{6},$ $e_{3}^{2} = -e_{6},$ $e_{4}^{2} = \alpha(e_{2} + e_{3}) + \beta e_{6},$ $e_{5}^{2} = \gamma e_{6},$ $e_{6}^{2} = 0$ avec $\alpha, \beta,\gamma \in K^{*}$} &   {$1$} &   {\small Non } \\\hline
  {$N_{6,22}(\alpha, \beta)$} &   {$e_{1}^{2} = e_{2} + e_{3},$ $e_{2}^{2} = e_{6},$ $e_{3}^{2} = -e_{6},$ $e_{4}^{2} = \alpha(e_{2} + e_{3}),$  $e_{5}^{2} = \beta(e_{2} + e_{3}),$ $e_{6}^{2} = 0$ avec $\alpha, \beta \in K^{*}$} &   {$1$} &   {\small Non}  \\\hline
  {$N_{6,23}(\alpha, \beta, \gamma, \delta)$} &   {$e_{1}^{2} = e_{2} + e_{3},$ $e_{2}^{2} = e_{6},$  $e_{3}^{2} = -e_{6},$ \newline $e_{4}^{2} = \alpha(e_{2} + e_{3}) + \beta e_{6},$ $e_{5}^{2} = \gamma(e_{2} + e_{3}) + \delta e_{6},$ $e_{6}^{2} = 0$ avec $\alpha\delta-\beta\gamma \neq 0$, $\al\g\ne0$ et $\beta\delta \neq 0$} &   {$1$} &   {\small Non}  \\\hline
  {$N_{6,24}(\alpha, \beta, \gamma)$} &   {$e_{1}^{2} = e_{2} + e_{3},$ $e_{2}^{2} = e_{6},$  $e_{3}^{2} = -e_{6},${} \newline $e_{4}^{2} = \alpha(e_{2} + e_{3}) + \beta e_{6},$ $e_{5}^{2} = \gamma(e_{2} + e_{3}),$ $e_{6}^{2} = 0$ avec $\al\beta\gamma \neq 0$} &   {$1$} &   {\small Non}  \\\hline
  {$N_{6,25}(\alpha)$} &   {$e_{1}^{2} = e_{2} + e_{3},$ $e_{2}^{2} = e_{5},$ $e_{3}^{2} = -e_{5},${} $e_{4}^{2} = \alpha(e_{2} + e_{3}) + e_{6},$  $e_{5}^{2} = e_{6}^{2} = 0$ avec $\alpha \in K^{*}$} &   {$2$} &   {\small Non}  \\\hline
  {$N_{6,26}$} &   {$e_{1}^{2} = e_{2} + e_{3} + e_{4},$ $e_{2}^{2} = e_{5},$  $e_{3}^{2} = e_{6},${} $e_{4}^{2} = -(e_{5} + e_{6}),$  $e_{5}^{2} = e_{6}^{2} = 0$} &   {$2$} &   {\small Non}  \\
\hline
\end{longtable}
}
\end{prop}

Soient $N$ une nil-alg\`{e}bre d'\'{e}volution \`{a} puissances associatives, ind\'{e}composable, de base naturelle $B = \{e_{i}$ ; $1 \leq i \leq n\}$ et $d : N \longrightarrow N$ une d\'{e}rivation. On pose $d(e_{i}) = \sum_{j = 1}^{n}d_{ji}e_{j}$, pour tout $1 \leq i \leq n$ et on s'int\'{e}resse aux d\'{e}rivations des nil-alg\`{e}bres d'\'{e}volution ind\'{e}composable de dimension au plus $6$ ; ainsi, $\dim(ann(N)) = 1$ ou~$2$.

\section{Nil-alg\`{e}bres associatives et ind\'{e}composables}
On suppose ici que $N$ est associative.

\subsection{Dimension de l'annulateur de $N$ est $1$}
On pose $ann(N) = Ke_{n}$, alors $e_{i}^{2} = \alpha_{i}e_{n}$, $e_{n}^{2} = 0$ avec $\alpha_{i} \neq 0$ $(1 \leq i \leq n - 1)$ (\cite[\S~4.1]{Ouatt2018}).

Soit $gl(n, K) = M_{n}(K) = (e_{ji})_{1 \leq i, j \leq n}$ avec $e_{ji}e_{lk} = \delta_{il}e_{jk}$. On pose $H = <h_{ji} = e_{ji} - \alpha_{i}^{-1}\alpha_{j}e_{ij}$ ; $1 \leq i < j \leq n - 1>$, $L = <e_{nj}$ ; $1 \leq j \leq n - 1>$ et $g = e_{11} + \cdots + e_{n - 1, n - 1} + 2e_{nn}$.

\begin{theo} Soit $N$ une nil-alg\`{e}bre d'\'{e}volution ind\'{e}composable associative, de dimension $n \geq 3$, telle que $\dim(ann(N)) = 1$. Alors $\D(N) = H\oplus L\oplus Kg$ et $\D(N) \simeq K^{\frac{n(n - 1)}{2}}\times K$.
\end{theo}

\begin{proof} Pour $i = 1, \ldots, n - 1$, on a $d(e_{i}^{2}) = 2e_{i}d(e_{i}) = 2d_{ii}e_{i}^{2} = 2\alpha_{i}d_{ii}e_{n}$ et $d(e_{i}^{2}) = \alpha_{i}d(e_{n})$ alors $d(e_{n}) = 2d_{ii}e_{n}$ car $\alpha_{i} \neq 0$. En particulier $d(e_{n}) = 2d_{11}e_{n}$, d'o\`{u} $d_{ii} = d_{11}$ pour $i \in \{1, \ldots, n - 1\}$.

Nous avons aussi : $0 = d(e_{i}e_{j}) =  d_{ij}e_{i}^{2} + d_{ji}e_{j}^{2} =  (d_{ij}\alpha_{i} + d_{ji}\alpha_{j})e_{n}$, donc $d_{ij}\alpha_{i} + d_{ji}\alpha_{j} = 0$, soit $d_{ji} = -\alpha_{i}\alpha_{j}^{-1}d_{ij}$ pour $1 \leq i \neq j \leq n - 1$.

Alors
$Mat_{B}(d) = \left(
                            \begin{array}{cccccc}
                                  d_{11} & -\alpha_{1}^{-1}\alpha_{2}d_{21} & -\alpha_{1}^{-1}\alpha_{3}d_{31} & \cdots & -\alpha_{1}^{-1}\alpha_{n -1}d_{n - 1, 1} & 0 \\
                                  d_{21} & d_{11} & -\alpha_{2}^{-1}\alpha_{3}d_{32} & \cdots & -\alpha_{2}^{-1}\alpha_{n -1}d_{n - 1, 2} & 0 \\
                                  d_{31} & d_{32} & d_{11} & \cdots & -\alpha_{3}^{-1}\alpha_{n -1}d_{n - 1, 3} & 0 \\
                                  \vdots & \vdots & \vdots & \ddots & \vdots & \vdots \\
                                  d_{n - 1, 1} & d_{n - 1, 2} & d_{n - 1, 3} & \cdots & d_{11} & 0 \\
                                  d_{n1} & d_{n2} & d_{n3} & \cdots & d_{n, n - 1} & 2d_{11} \\
                                \end{array}
                              \right)$

Donc $Mat_{B}(d) = \sum_{1 \leq i < j \leq n - 1}d_{ji}h_{ji} + \sum_{j = 1}^{n - 1}d_{nj}e_{nj} + d_{11}g$, d'o\`{u} $\D(N) = H\oplus L\oplus Kg$.
On a $\dim(H) = \frac{(n - 2)(n - 1)}{2}$, $\dim(L) = n - 1$ et $\dim(Kg) = 1$. Ainsi $\D(N) \simeq K^{\frac{n(n - 1)}{2}}\times K$.
\end{proof}
\begin{prop} Soit $N$ une nil-alg\`{e}bre d'\'{e}volution ind\'{e}composable associative de dimension $n \geq 3$ telle que $\dim(ann(N)) = 1$. Alors $\D(N)' = H\oplus L$ et $\D(N)'~\simeq~K^{\frac{n(n - 1)}{2}}$.
\end{prop}

\begin{proof} $\bullet$ Calculons $[H, H]$. Soient $1 \leq i < j \leq n - 1$ et $1 \leq k < l \leq n - 1,$ on a
\begin{eqnarray*}
[h_{ji}, h_{lk}] & = &  (e_{ji} - \alpha_{i}^{-1}\alpha_{j}e_{ij})(e_{lk} - \alpha_{k}^{-1}\alpha_{l}e_{kl}) - (e_{lk} - \alpha_{k}^{-1}\alpha_{l}e_{kl})(e_{ji} - \alpha_{i}^{-1}\alpha_{j}e_{ij}) \\
 & = & (\delta_{il}e_{jk} - \alpha_{k}^{-1}\alpha_{l}\delta_{ik}e_{jl} - \alpha_{i}^{-1}\alpha_{j}\delta_{jl}e_{ik} + \alpha_{i}^{-1}\alpha_{j}\alpha_{k}^{-1}\alpha_{l}\delta_{jk}e_{il})\nonumber\\
& & {} - (\delta_{kj}e_{li} - \alpha_{i}^{-1}\alpha_{j}\delta_{ki}e_{lj} - \alpha_{k}^{-1}\alpha_{l}\delta_{lj}e_{ki} + \alpha_{k}^{-1}\alpha_{j}\alpha_{i}^{-1}\alpha_{l}\delta_{li}e_{kj})
\end{eqnarray*}

$1)$ $j = k$, on a $[h_{ji}, h_{lj}] = -(e_{li} - \alpha_{i}^{-1}\alpha_{l}e_{il}) = -h_{li}$.

$2)$ $j = l$ et $i < k$, on a $[h_{ji}, h_{jk}] = \alpha_{k}^{-1}\alpha_{j}(e_{ki} - \alpha_{i}^{-1}\alpha_{k}e_{ik}) = \alpha_{k}^{-1}\alpha_{j}h_{ki}$.

$3)$ $j = l$ et $k < i$, on a $[h_{ji}, h_{jk}] = -\alpha_{i}^{-1}\alpha_{j}(e_{ik} - \alpha_{k}^{-1}\alpha_{i}e_{ki}) = -\alpha_{i}^{-1}\alpha_{j}h_{ik}$.

$4)$ $i = l$, on a $[h_{ji}, h_{ik}] = (e_{jk} - \alpha_{k}^{-1}\alpha_{j}e_{kj}) = h_{jk}$.

$5)$ $i = k$ et $j < l$, on a $[h_{ji}, h_{li}] = \alpha_{i}^{-1}\alpha_{j}(e_{lj} - \alpha_{j}^{-1}\alpha_{l}e_{jl}) = \alpha_{i}^{-1}\alpha_{j}h_{lj}$.

$6)$ $i = k$ et $l < j$, on a $[h_{ji}, h_{li}] = -\alpha_{i}^{-1}\alpha_{l}(e_{jl} - \alpha_{l}^{-1}\alpha_{j}e_{lj}) = -\alpha_{i}^{-1}\alpha_{l}h_{jl}$.

$7)$ $i, j, k, l$ sont deux \`{a} deux distincts, on a $[h_{ji}, h_{lk}] = 0$.

On en d\'{e}duit que $[H, H] \subseteq H.$ Posons $[H, H]_{1} = <[h_{j1}, h_{l1}] = \alpha_{1}^{-1}\alpha_{j}h_{lj}$ ; $1 < j < l \leq n - 1>$ et $[H, H]_{2} = <[h_{n - 1, 1}, h_{n - 1, k}] = \alpha_{k}^{-1}\alpha_{n - 1}h_{k1}$ ; $2 \leq k \leq n - 2>$. On a $\dim([H, H]_{1}) = (n -~3) + (n - 4) + \ldots + 1 = \frac{(n - 3)(n - 2)}{2}$, $\dim([H, H]_{2}) = n - 3$ et $[h_{21}, h_{n - 1, 2}] = -h_{n - 1, 1}$. Comme $[H, H]_{1}\oplus[H, H]_{2}\oplus K[h_{21}, h_{n - 1, 2}] \subseteq [H, H]$ et que $\dim([H, H]_{1}\oplus[H, H]_{2}\oplus K[h_{21}, h_{n - 1, 2}]) = \frac{(n - 3)(n - 2)}{2} + (n - 3) + 1 = \frac{(n - 1)(n - 2)}{2} = \dim(H)$ alors $[H, H] = H.$

$\bullet$ Calculons $[H, L]$. Soient $1 \leq i < j \leq n - 1$ et $1 \leq k \leq n - 1$. On a\\
$[h_{ji}, e_{nk}] = (e_{ji} -~\alpha_{i}^{-1}\alpha_{j}e_{ij})e_{nk} - e_{nk}(e_{ji} - \alpha_{i}^{-1}\alpha_{j}e_{ij}) = -\delta_{kj}e_{ni} + \alpha_{i}^{-1}\alpha_{j}\delta_{ki}e_{nj}$.

$1)$ $k = i$ entra\^{\i}ne $[h_{ji}, e_{ni}] = \alpha_{i}^{-1}\alpha_{j}e_{nj}$.

$2)$ $k = j$ entra\^{\i}ne  $[h_{ji}, e_{nj}] = -e_{ni}$.

$3)$ $k \neq i$ et $k \neq j$ entra\^{\i}nent $[h_{ji}, e_{nk}] = 0$.

On en d\'{e}duit que $[H, L] \subseteq L.$ Posons $[H, L]_{1} = <[h_{n -1, i}, e_{n, n -1}] = -e_{ni}$ ; $1 \leq i \leq n - 2>$~; on a $[h_{n - 1, 1}, e_{n1}] = \alpha_{1}^{-1}\alpha_{n - 1}e_{n, n - 1}$. Comme $[H, L]_{1}\oplus K[h_{n - 1, 1}, e_{n1}] \subseteq [H, L]$ et $\dim([H, L]_{1}\oplus K[h_{n - 1, 1}, e_{n1}]) = (n - 2) + 1 = n - 1 = \dim(L)$, alors $[H, L] = L$.

$\bullet$ Calculons $[H, Kg]$. Soient $1 \leq i < j \leq n - 1$. On a $[h_{ji}, g] = \sum_{k = 1}^{n - 1}[h_{ji}, e_{kk}] + 2[h_{ji}, e_{nn}] = \sum_{k = 1}^{n - 1}((e_{ji} - \alpha_{i}^{-1}\alpha_{j}e_{ij})e_{kk} - e_{kk}(e_{ji} - \alpha_{i}^{-1}\alpha_{j}e_{ij})) = (e_{ji} - \alpha_{i}^{-1}\alpha_{j}e_{ij}) - (e_{ji} - \alpha_{i}^{-1}\alpha_{j}e_{ij})$. On en d\'{e}duit que $[H, Kg] = 0$.

$\bullet$ Calculons $[L, L]$. Soient $1 \leq i, j \leq n - 1$, on a $[e_{ni}, e_{nj}] = e_{ni}e_{nj} - e_{nj}e_{ni} = \delta_{in}e_{nj} - \delta_{jn}e_{ni} = 0$. Donc $[L, L] = 0$.

$\bullet$ Calculons $[L, Kg]$. Soient $1 \leq i \leq n - 1$, on a $[e_{ni}, g] = \sum_{j = 1}^{n}e_{ni}e_{jj} - \sum_{j = 1}^{n}e_{jj}e_{ni} = e_{ni}$. Donc $[L, Kg] = L$

$\bullet$ On a $[Kg, Kg] = 0$.

On en d\'{e}duit que  $\D(N)' = H\oplus L$ et $\D(N)'~\simeq~K^{\frac{n(n - 1)}{2}}$.
\end{proof}

\subsection{Dimension de l'annulateur de $N$ est $2$}

\begin{theo} Soit $N$ une nil-alg\`{e}bre d'\'{e}volution ind\'{e}composable, associative, de dimension $\leq 6$ telle que $\dim(ann(N)) = 2$. Alors, les d\'{e}rivations dans $N$ sont donn\'ees dans la Table $4$.

 \begin{longtable}[c]{|p{2,5cm}|p{9cm}|c|}
 \caption*{\underline{\textbf{Table 4}} }\\
  \hline
    {$N$} &   {\small D\'{e}finition de la d\'{e}rivation} &   {\small $\D(N)\simeq$}   \\\hline
\endfirsthead
 \hline   {$N$} &   {\small D\'{e}rivation} &   {\small $\D(N)\simeq$}   \\\hline
\endhead \hline
\endfoot
{$N_{5,9}(\alpha, \beta)$} & {$d(e_{i}) = d_{11}e_{i} + d_{4i}e_{4} + d_{5i}e_{5}$, $d(e_{4}) = 2d_{11}e_{4}$, $d(e_{5}) = 2d_{11}e_{5}$ avec $i = 1, 2, 3$} & {$K^{4}\times K^3$}   \\\hline
{$N_{6,17}(\alpha, \beta, \gamma)$} & {$d(e_{i}) = d_{11}e_{i} + d_{5i}e_{5} + d_{6i}e_{6}$, $d(e_{3}) = d_{11}e_{3} + d_{43}e_{4} + d_{53}e_{5} + d_{63}e_{6}$, $d(e_{4}) = -\gamma^{-1}d_{43}e_{3} + d_{11}e_{4} + d_{54}e_{5} + d_{64}e_{6}$, $d(e_{5}) = 2d_{11}e_{5}$, $d(e_{6}) = 2d_{11}e_{6}$ avec $i = 1, 2$} & $K^{6}\times K^4$ \\\hline
{$N_{6,18}(\alpha, \beta, \gamma, \delta)$} & {$d(e_{i}) = d_{11}e_{i} + d_{5i}e_{5} + d_{6i}e_{6}$, $d(e_{5}) = 2d_{11}e_{5}$, $d(e_{6}) = 2d_{11}e_{6}$ avec $i = 1, 2, 3, 4$} & $K^{5}\times K^4$ \\\hline
\end{longtable}
\end{theo}

\begin{proof}
Dans \cite{Ouatt2018}, les Auteurs montrent que $N$ admet une base naturelle $B = \{e_{i}$ ; $1 \leq i \leq n\}$ dont la table de multiplication est d\'{e}finie par : $e_{1}^{2} = e_{n - 1}$, $e_{i}^{2} = \alpha_{i, n - 1}e_{n - 1} + \alpha_{i, n}e_{n}$, $e_{n - 2}^{2} = e_{n}$, $e_{n - 1}^{2} = e_{n}^{2} = 0$ avec $(\alpha_{i, n -1}, \alpha_{i, n}) \neq 0$ o\`{u} $2  \leq i \leq n - 3$. Il existe \'egalement un $i_{0} \in \{2, \ldots, n - 3\}$ tel que $\alpha_{i_{0}, n -1}\alpha_{i_{0}, n} \neq 0$. En d\'{e}rivant $e_{n - 1} = e_{1}^{2}$, on obtient $d(e_{n - 1}) = d(e_{1}^{2}) = 2e_{1}d(e_{1}) = 2d_{11}e_{1}^{2} = 2d_{11}e_{n - 1}$. De m\^{e}me $d(e_{n}) = 2d_{n -2, n -2}e_{n}$ en d\'{e}rivant $e_{n} = e_{n - 2}^{2}$. Pour $i = 2, \ldots, n - 3$, $d(e_{i}^{2}) = 2e_{i}d(e_{i}) = 2d_{ii}e_{i}^{2} = 2d_{ii}(\alpha_{i, n - 1}e_{n - 1} + \alpha_{i, n}e_{n})$ et la d\'{e}rivation de $\alpha_{i, n - 1}e_{n - 1} + \alpha_{i, n}e_{n}$ donne $\alpha_{i, n - 1}d(e_{n - 1}) + \alpha_{i, n}d(e_{n}) = 2(\alpha_{i, n - 1}d_{11}e_{n - 1} + \alpha_{i, n}d_{n - 2, n - 2}e_{n})$. Comme $\alpha_{i, n - 1}e_{n - 1} + \alpha_{i, n}e_{n} = e_{i}^{2}$, alors $d_{ii}\alpha_{i, n - 1} = \alpha_{i, n - 1}d_{11} \textrm{ et }d_{ii}\alpha_{i, n} = \alpha_{i, n}d_{n - 2, n - 2}$. Pour $i = i_{0}$, on a $d_{11} = d_{i_{0}i_{0}} = d_{n - 2, n - 2}$ car $\alpha_{i_{0}, n -1}\alpha_{i_{0}, n} \neq 0$ et pour $i \neq i_{0}$, on a $d_{ii} = d_{11}$ car $(\alpha_{i, n -1}, \alpha_{i, n}) \neq 0$. On en d\'{e}duit que
\begin{eqnarray}\label{Eq5}
d(e_{n - 1}) = 2d_{11}e_{n - 1},\textrm{ }d(e_{n}) = 2d_{11}e_{n}\textrm{ et }d_{ii} = d_{11}\textrm{ avec }  1 \leq i \leq n - 2.
\end{eqnarray}
Les r\'{e}lations $(a)$, $(b)$ et \eqref{Eq5} sont suffisantes pour caract\'{e}riser les d\'{e}rivations dans $N$.

$\bullet$ $N_{5, 9}(\alpha, \beta)$ : $e_{1}^{2} = e_{4}$, $e_{2}^{2} = \alpha e_{4} + \beta e_{5}$, $e_{3}^{2} = e_{5}$, $e_{4}^{2} = e_{5}^{2} = 0$ avec $\alpha\beta \neq 0$. $(a)$ entra\^{\i}ne $d_{12} = d_{21} = 0$, $d_{13} = d_{31} = 0$ et $d_{23} = d_{32} = 0$ ; \eqref{Eq5} entra\^{\i}ne $d(e_{4}) = 2d_{11}e_{4}$, $d(e_{5}) = 2d_{11}e_{5}$ et $d_{11} = d_{22} = d_{33}$. Donc $d(e_{1}) = d_{11}e_{1} + d_{41}e_{4} + d_{51}e_{5}$, $d(e_{2}) = d_{11}e_{2} + d_{42}e_{4} + d_{52}e_{5}$ et $d(e_{3}) = d_{11}e_{3} + d_{43}e_{4} + d_{53}e_{5}$.

$\bullet$ $N_{6, 17}(\alpha, \beta, \gamma)$ : $e_{1}^{2} = e_{5}$, $e_{2}^{2} = \alpha e_{5} + \beta e_{6}$, $e_{3}^{2} = \gamma e_{6}$, $e_{4}^{2} = e_{6}$, $e_{5}^{2} = e_{6}^{2} = 0$ avec~$\alpha\beta\gamma \neq 0$. $(a)$ entra\^{\i}ne $d_{12} = d_{21} = 0$, $d_{13} = d_{31} = 0$, $d_{14} = d_{41} = 0$, $d_{23} = d_{32} = 0$ et $d_{24} = d_{42} = 0$ ; $(b)$ entra\^{\i}ne $d_{34} = -\gamma^{-1}d_{43}$ ;
\eqref{Eq5} entra\^{\i}ne $d(e_{5}) = 2d_{11}e_{5}$, $d(e_{6}) = 2d_{11}e_{6}$ et $d_{11} = d_{22} = d_{33} = d_{44}$.
Donc $d(e_{1}) = d_{11}e_{1} + d_{51}e_{5} + d_{61}e_{6}$, $d(e_{2}) = d_{11}e_{2} + d_{52}e_{5} + d_{62}e_{6}$, $d(e_{3}) = d_{11}e_{3} + d_{43}e_{4} + d_{53}e_{5} + d_{63}e_{6}$ et $d(e_{4}) = -\gamma^{-1}d_{43}e_{3} + d_{11}e_{4} + d_{54}e_{5} + d_{64}e_{6}$.

$\bullet$ $N_{6, 18}(\alpha, \beta, \gamma, \delta)$ : $e_{1}^{2} = e_{5}$, $e_{2}^{2} = \alpha e_{5} + \beta e_{6}$, $e_{3}^{2} = \gamma e_{5} + \delta e_{6}$, $e_{4}^{2} = e_{6}$, $e_{5}^{2} = e_{6}^{2} = 0$ avec $\alpha\beta \neq 0$, $\gamma\delta \neq 0$ et $\alpha\delta - \beta\gamma \neq 0$. $(a)$ entra\^{\i}ne $d_{12} = d_{21} = 0$, $d_{13} = d_{31} = 0$, $d_{14} = d_{41} = 0$, $d_{23} = d_{32} = 0$, $d_{24} = d_{42} = 0$ et $d_{34} = d_{43} = 0$ ; \eqref{Eq5} entra\^{\i}ne $d(e_{5}) = 2d_{11}e_{5}$, $d(e_{6}) = 2d_{11}e_{6}$ et $d_{11} = d_{22} = d_{33} = d_{44}$. Donc  $d(e_{1}) = d_{11}e_{1} + d_{51}e_{5} + d_{61}e_{6}$, $d(e_{2}) = d_{11}e_{2} + d_{52}e_{5} + d_{62}e_{6}$, $d(e_{3}) = d_{11}e_{3} +  d_{53}e_{5} + d_{63}e_{6}$ et $d(e_{4}) = d_{11}e_{4} + d_{54}e_{5} + d_{64}e_{6}$.
\end{proof}

\begin{prop} Soit $N$ une nil-alg\`{e}bre d'\'{e}volution ind\'{e}composable, associative, de dimension $\leq 6$ telle que $\dim(ann(N)) = 2$. Alors, le crochet de deux d\'{e}rivations dans $N$ est donn\'e dans la Table~$5$.

 \begin{longtable}[c]{|p{2,4cm}|p{8,5cm}|c|}
 \caption*{\underline{\textbf{Table 5}} }\\
  \hline
    {$N$} &   {\small D\'{e}finition de la d\'{e}rivation $[d, d']$} &   {\small $\D(N)'\simeq$}   \\\hline
\endfirsthead
 \hline   {$N$} &   {\small D\'{e}finition de la d\'{e}rivation $[d, d']$} &   {\small $\D(N)'\simeq$}   \\\hline
\endhead \hline
\endfoot
{$N_{5,9}(\alpha, \beta)$} & {$[d,d']_{11} =  0$, $[d,d']_{4i} = d_{4i}'d_{11} - d_{4i}d_{11}'$, $[d,d']_{5i} = d_{5i}'d_{11} - d_{5i}d_{11}'$  avec $i = 1, 2, 3$} & {$K^{3}\times K^3$}   \\\hline
{$N_{6,17}(\alpha, \beta, \gamma)$} & {$[d,d']_{11} = [d,d']_{43} = 0$, $[d,d']_{ji} = d_{ji}'d_{11} - d_{ji}d_{11}'$, $[d,d']_{j3} = (d_{j3}'d_{11} - d_{j3}d_{11}') + (d_{43}'d_{j4} - d_{43}d_{j4}')$, $[d,d']_{j4} = (d_{j4}'d_{11} - d_{j4}d_{11}') + \gamma^{-1}(d_{43}d_{j3}' - d_{43}'d_{j3})$ avec $i = 1, 2$ et $j = 5, 6$} & $K^{4}\times K^4$ \\\hline
{$N_{6,18}(\alpha, \beta, \gamma, \delta)$} & {$[d,d']_{11} = 0$, $[d,d']_{ji} = d_{ji}'d_{11} - d_{ji}d_{11}'$ avec $i = 1, 2, 3, 4$ et $j = 5, 6$} & $K^{4}\times K^4$ \\\hline
\end{longtable}
\end{prop}

\begin{proof} Soient $d, d'$ deux d\'{e}rivations dans $N$. Calculons $[d,d']$.

$\bullet$ $N_{5, 9}(\alpha, \beta)$: Pour $1 \leq j \leq 3$, on a $d\circ d' (e_{j}) = d_{11}'d(e_{j}) + d_{4j}'d(e_{4}) + d_{5j}'d(e_{5})$. Alors $[d, d']_{11} = d_{11}'d_{11} - d_{11}d_{11}' = 0$, $[d, d']_{4j} = (d_{11}'d_{4j} + 2d_{4j}'d_{11}) - (d_{11}d_{4j}' + 2d_{4j}d_{11}') = d_{4j}'d_{11} - d_{4j}d_{11}'$ et $[d, d']_{5j} = (d_{11}'d_{5j} + 2d_{5j}'d_{11}) - (d_{11}d_{5j}' + 2d_{5j}d_{11}') = d_{5j}'d_{11} - d_{5j}d_{11}'$. Donc $\D(N_{5, 9}(\alpha, \beta))'\simeq~K^{3}\times K^{3}$.

$\bullet$ $N_{6, 17}(\alpha, \beta, \gamma)$: Pour $1 \leq j \leq 2$, on a $d\circ d' (e_{j}) = d_{11}'d(e_{j}) + d_{5j}'d(e_{5}) + d_{6j}'d(e_{6})$, alors $[d, d']_{11} = d_{11}'d_{11} - d_{11}d_{11}' = 0$, $[d, d']_{5j} = (d_{11}'d_{5j} + 2d_{5j}'d_{11}) - (d_{11}d_{5j}' + 2d_{5j}d_{11}') = d_{5j}'d_{11} - d_{5j}d_{11}'$ et $[d, d']_{6j} = (d_{11}'d_{6j} + 2d_{6j}'d_{11}) - (d_{11}d_{6j}' + 2d_{6j}d_{11}') = d_{6j}'d_{11} - d_{6j}d_{11}'$. On a $d\circ d' (e_{3}) = d_{11}'d(e_{3}) + d_{43}'d(e_{4}) + d_{53}'d(e_{5}) + d_{63}'d(e_{6})$, alors $[d, d']_{43} = (d_{11}'d_{43} + d_{43}'d_{11}) - (d_{11}d_{43}' + d_{43}d_{11}') = 0$, $[d, d']_{53} = (d_{11}'d_{53} + d_{43}'d_{54} + 2d_{53}'d_{11}) - (d_{11}d_{53}' + d_{43}d_{54}' + 2d_{53}d_{11}') = (d_{53}'d_{11} - d_{53}d_{11}') + (d_{43}'d_{54} - d_{43}d_{54}')$ et $[d, d']_{63} = (d_{11}'d_{63} + d_{43}'d_{64} + 2d_{63}'d_{11}) - (d_{11}d_{63}' + d_{43}d_{64}' + 2d_{63}d_{11}') = (d_{63}'d_{11} - d_{63}d_{11}') + (d_{43}'d_{64} - d_{43}d_{64}')$.
On a $d\circ d' (e_{4}) = -\gamma^{-1}d_{43}'d(e_{3}) + d_{11}'d(e_{4}) + d_{54}'d(e_{5}) + d_{64}'d(e_{6})$, alors $[d, d']_{54} = (-\gamma^{-1}d_{43}'d_{53} + d_{11}'d_{54} + 2d_{54}'d_{11}) - (-\gamma^{-1}d_{43}d_{53}' + d_{11}d_{54}' + 2d_{54}d_{11}') = (d_{54}'d_{11} - d_{54}d_{11}') + \gamma^{-1}(d_{53}'d_{43} - d_{53}d_{43}')$ et $[d, d']_{64} = (-\gamma^{-1}d_{43}'d_{63} + d_{11}'d_{64} + 2d_{64}'d_{11}) - (-\gamma^{-1}d_{43}d_{63}' + d_{11}d_{64}' + 2d_{64}d_{11}') = (d_{64}'d_{11} - d_{64}d_{11}') + \gamma^{-1}(d_{63}'d_{43} - d_{63}d_{43}')$. Donc $\D(N_{6, 17}(\alpha, \beta, \gamma))' \simeq~K^{4}\times K^{4}$.

$\bullet$ $N_{6, 18}(\alpha, \beta, \gamma, \delta)$: Pour $1 \leq j \leq 4$, on a $d\circ d' (e_{j}) = d_{11}'d(e_{j}) + d_{5j}'d(e_{5}) + d_{6j}'d(e_{6})$, alors $[d, d']_{11} = d_{11}'d_{11} - d_{11}d_{11}' =~0$, $[d, d']_{5j} = (d_{11}'d_{5j} + 2d_{5j}'d_{11}) - (d_{11}d_{5j}' + 2d_{5j}d_{11}') = d_{5j}'d_{11} - d_{5j}d_{11}'$ et $[d, d']_{6j} = (d_{11}'d_{6j} + 2d_{6j}'d_{11}) - (d_{11}d_{6j}' + 2d_{6j}d_{11}') = d_{6j}'d_{11} - d_{6j}d_{11}'$. Donc $\D(N_{6, 18}(\alpha, \beta, \gamma, \delta))'\simeq~K^{4}\times K^{4}$.
\end{proof}

\begin{prop} Soient $N$ une nil-alg\`{e}bre d'\'{e}volution ind\'{e}composable associative de dimension $n \geq 3$ telle que $\dim(ann(N)) = 1$ et $a=\sum_{i=1}^na_ie_i\in N$. Alors $R_a$ est une d\'erivation si et seulement si $R_a=\sum_{i=1}^{n-1}\al_ia_ie_{ni}$. De plus $\In(N) = L\simeq K^{n-1}$.
\end{prop}

\begin{proof} L'alg\`ebre $N$ \'etant commutative, les seules d\'erivations int\'erieures de $N$ sont de la forme $R_a$, avec $a=\sum_{i=1}^na_ie_i\in N$. On a $R_a(e_i)=\al_i a_ie_n$ et $R_a(e_n)=0$, d'o\`u le r\'esultat.
\end{proof}

\begin{prop} Soit $N$ une nil-alg\`{e}bre d'\'{e}volution ind\'{e}composable associative, de dimension $\leq 6$, telle que $\dim(ann(N)) = 2$. La Table~$6$ caract\'{e}rise les d\'{e}rivations int\'{e}rieures de $N$.

\begin{longtable}[c]{|p{2.5cm}|p{9cm}|c|}
 \caption*{\underline{\textbf{Table 6}} }\\
  \hline
    {$N$} &   {\small $R_{c}\in \D(N)$ si et seulement si} &   {\small $\In(N) \simeq$}   \\\hline
\endfirsthead
 \hline   {$N$} &   {\small $R_{c} \in \D(N)$ si et seulement si} &   {\small $\In(N) \simeq$}   \\\hline
\endhead \hline
{$N_{5, 9}(\alpha,\be)$} & {$R_c=c_1e_{41}+c_2(\al e_{42}+\be e_{52})+c_3e_{53}$}  & {$K^{3}$}   \\\hline
{$N_{6, 17}(\al,\be,\g)$} & {$R_c=c_1e_{51}+c_2(\al e_{52}+\be e_{62})+c_3\g e_{63}+c_4e_{64}$} & {$K^4$}   \\\hline
{$N_{6, 18}(\al,\be,\g,\de)$} & {$R_c=c_1e_{51}+c_2(\al e_{52}+\be e_{62})+c_3(\g e_{53}+\delta c_4e_{63})+c_4e_{64}$} & {$K^4$}   \\\hline
\end{longtable}
\end{prop}

\begin{proof}
Les d\'erivations int\'erieures de $N$ sont de la forme $R_c$, avec $c=\sum_{i=1}^nc_ie_i\in N$.

$\bullet$ $N_{5, 9}(\alpha,\be)$ : on a $R_c(e_1) = c_1e_1^{2} = c_1e_4$, $R_c(e_2) = c_2e_2^{2} = c_2(\alpha e_4 + \be e_5)$, $R_c(e_3) = c_3e_3^{2} = c_3e_5$, $R_c(e_4) = R_c(e_5) = 0$. Les \'{e}l\'{e}ments de $\In(N_{5, 9}(\alpha,\be))$ sont de la forme $R_c = c_1e_{41} + c_2(\alpha e_{42} + \be e_{52}) + c_3e_{53}$ et $\In(N_{5, 9}(\alpha,\be)) \simeq~K^{3}$.

$\bullet$ $N_{6, 17}(\alpha,\be,\gamma)$ : on a $R_c(e_1) = c_1e_1^{2} = c_1e_5$, $R_c(e_2) = c_2e_2^{2} = c_2(\alpha e_5 + \be e_6)$, $R_c(e_3) = c_3e_3^{2} = c_3\gamma e_6$, $R_c(e_4) = c_4e_4^{2} = c_4e_6$, $R_c(e_5) = R_c(e_6) = 0$. Les \'{e}l\'{e}ments de $\In(N_{6, 17}(\alpha,\be,\gamma))$ sont de la forme $R_c = c_1e_{51} + c_2(\alpha e_{52} + \be e_{62}) + c_3\gamma e_{63} + c_4e_{64}$ et \\
$\In(N_{6, 17}(\alpha,\be,\gamma)) \simeq~K^{4}$.

$\bullet$ $N_{6, 18}(\alpha,\be,\gamma, \delta)$ : on a $R_c(e_1) = c_1e_1^{2} = c_1e_5$, $R_c(e_2) = c_2e_2^{2} = c_2(\alpha e_5 + \be e_6)$, $R_c(e_3) = c_3e_3^{2} = c_3(\gamma e_5 + \delta e_6)$, $R_c(e_4) = c_4e_4^{2} = c_4e_6$, $R_c(e_5) = R_c(e_6) = 0$. Les \'{e}l\'{e}ments de $\In(N_{6,18}(\alpha,\be,\gamma, \delta))$ sont de la forme $R_c = c_1e_{51} + c_2(\alpha e_{52} + \be e_{62}) + c_3(\gamma e_{53} + \delta e_{63}) + c_4e_{64}$ et $\In(N_{6,18}(\alpha,\be,\gamma, \delta)) \simeq~K^{4}$.
\end{proof}

\section{Nil-alg\`{e}bres ind\'{e}composables qui n'est pas associatives}

Ici, les d\'erivations int\'erieures sont de la forme $R_c+[R_a,R_b]$.

\subsection{Dimension de l'annulateur de $N$ est $1$}

\begin{theo}\label{DerNonAss} Soit $N$ une nil-alg\`{e}bre d'\'{e}volution ind\'{e}composable, \`{a} puissances associatives, de dimension $\leq 6$, qui n'est pas associative telle que $\dim(ann(N)) = 1$. Alors, la Table $7$ caract\'{e}rise les d\'{e}rivations dans $N$.

\begin{longtable}[c]{|p{2,4cm}|p{8.7cm}|c|}
 \caption*{\underline{\textbf{Table 7}} }\\
  \hline
    {$N$} &   {\small D\'{e}finition de la d\'{e}rivation} &   {\small $\D(N)\simeq$}   \\\hline
\endfirsthead
 \hline   {$N$} &   {\small D\'{e}rivation} &   {\small $\D(N)\simeq$}   \\\hline
\endhead \hline
\endfoot
{$N_{4,6}$} & {$d(e_{1}) = d_{11}e_{1} + d_{41}e_{4}$, $d(e_{2}) = d_{22}e_{2} + (2d_{11} - d_{22})e_{3} + d_{42}e_{4}$, $d(e_{3}) = (2d_{11} - d_{22})e_{2} + d_{22}e_{3} - d_{42}e_{4}$, $d(e_{4}) = 2d_{22}e_{4}$} & {$K^{2}\times K^2$}   \\\hline
{$N_{5,10}(\alpha)$} & {$d(e_{1}) = d_{11}e_{1} + d_{51}e_{5}$, $d(e_{2}) = d_{22}e_{2} + (2d_{11} - d_{22})e_{3} + d_{42}e_{4} + d_{52}e_{5}$, $d(e_{3}) = (2d_{11} - d_{22})e_{2} + d_{22}e_{3} - d_{42}e_{4} - d_{52}e_{5}$, $d(e_{4}) = -\alpha d_{42}e_{2} - \alpha d_{42}e_{3} + d_{22}e_{4} + d_{54}e_{5}$, $d(e_{5}) = 2d_{22}e_{5}$} & $K^{3}\times K^3$ \\\hline
{$N_{5,11}(\alpha)$} & {$d(e_{1}) = d_{11}e_{1} + d_{41}e_{4} + d_{51}e_{5}$, $d(e_{2}) = d_{22}e_{2} + (2d_{11} - d_{22})e_{3} + d_{52}e_{5}$, $d(e_{3}) = (2d_{11} - d_{22})e_{2} + d_{22}e_{3} - d_{52}e_{5}$, $d(e_{4}) = -\alpha d_{41}e_{1} + d_{11}e_{4} + d_{54}e_{5}$, $d(e_{5}) = 2d_{22}e_{5}$} & $K^{3}\times K^3$ \\\hline
{$N_{5,12}(\alpha, \beta)$} & {$d(e_{1}) = d_{11}e_{1} + d_{51}e_{5}$, $d(e_{2}) = d_{11}e_{2} + d_{11}e_{3} + d_{52}e_{5}$, $d(e_{3}) = d_{11}e_{2} + d_{11}e_{3} - d_{52}e_{5}$, $d(e_{4}) = d_{11}e_{4} + d_{54}e_{5}$, $d(e_{5}) = 2d_{11}e_{5}$} & $K\times K^3$ \\\hline
{$N_{6,19}(\alpha, \beta)$} & {$d(e_{1}) = d_{11}e_{1} + d_{61}e_{6}$, $d(e_{2}) = d_{22}e_{2} + (2d_{11} - d_{22})e_{3} + d_{42}e_{4} + d_{52}e_{5} + d_{62}e_{6}$, $d(e_{3}) = (2d_{11} - d_{22})e_{2} + d_{22}e_{3} - d_{42}e_{4} - d_{52}e_{5} - d_{62}e_{6}$, $d(e_{4}) = -\alpha d_{42}e_{2} - \alpha d_{42}e_{3} + d_{22}e_{4} + d_{54}e_{5} + d_{64}e_{6}$, $d(e_{5}) = -\beta d_{52}e_{2} - \beta d_{52}e_{3} - \alpha^{-1}\beta d_{54}e_{4} + d_{22}e_{5} + d_{65}e_{6}$, $d(e_{6}) = 2d_{22}e_{6}$} & $K^{5}\times K^4$ \\\hline
{$N_{6,20}(\alpha, \beta)$} & {$d(e_{1}) = d_{11}e_{1} + d_{41}e_{4} + d_{61}e_{6}$, $d(e_{2}) = d_{22}e_{2} + (2d_{11} - d_{22})e_{3} + d_{52}e_{5} + d_{62}e_{6}$, $d(e_{3}) = (2d_{11} - d_{22})e_{2} + d_{22}e_{3} - d_{52}e_{5} - d_{62}e_{6}$, $d(e_{4}) = -\alpha d_{41}e_{1} + d_{11}e_{4} + d_{64}e_{6}$, $d(e_{5}) = -\beta d_{52}e_{2} - \beta d_{52}e_{3} + d_{22}e_{5} + d_{65}e_{6}$, $d(e_{6}) = 2d_{22}e_{6}$} & $K^{4}\times K^4$ \\\hline
{$N_{6,21}(\alpha, \beta, \gamma)$} & {$d(e_{1}) = d_{11}e_{1} + d_{61}e_{6}$, $d(e_{2}) = d_{11}e_{2} + d_{11}e_{3} + d_{52}e_{5} + d_{62}e_{6}$, $d(e_{3}) = d_{11}e_{2} + d_{11}e_{3} - d_{52}e_{5} - d_{62}e_{6}$, $d(e_{4}) = d_{11}e_{4} + d_{64}e_{6}$, $d(e_{5}) = -\gamma d_{52}e_{2} - \gamma d_{52}e_{3} + d_{11}e_{5} + d_{65}e_{6}$, $d(e_{6}) = 2d_{11}e_{6}$} & $K^{2}\times K^4$ \\\hline
{$N_{6,22}(\alpha, \beta)$} & {$d(e_{1}) = d_{11}e_{1} + d_{41}e_{4} + d_{51}e_{5} + d_{61}e_{6}$, $d(e_{2}) = d_{22}e_{2} + (2d_{11} - d_{22})e_{3} + d_{62}e_{6}$, $d(e_{3}) = (2d_{11} - d_{22})e_{2} + d_{22}e_{3} - d_{62}e_{6}$, $d(e_{4}) = -\alpha d_{41}e_{1} + d_{11}e_{4} + d_{54}e_{5} + d_{64}e_{6}$, $d(e_{5}) = -\beta d_{51}e_{1} - \alpha^{-1}\beta d_{54}e_{4} + d_{11}e_{5} + d_{65}e_{6}$, $d(e_{6}) = 2d_{22}e_{6}$} & $K^{5}\times K^4$ \\\hline
{$N_{6,23}(\alpha, \beta, \gamma, \delta)$} & {$d(e_{1}) = d_{11}e_{1} + d_{61}e_{6}$, $d(e_{2}) = d_{11}e_{2} + d_{11}e_{3} + d_{62}e_{6}$, $d(e_{3}) = d_{11}e_{2} + d_{11}e_{3} - d_{62}e_{6}$, $d(e_{4}) = d_{11}e_{4} + d_{64}e_{6}$, $d(e_{5}) = d_{11}e_{5} + d_{65}e_{6}$, $d(e_{6}) = 2d_{11}e_{6}$} & $K\times K^4$ \\\hline
{$N_{6,24}(\alpha, \beta, \gamma)$} & {$d(e_{1}) = d_{11}e_{1} + d_{51}e_{5} + d_{61}e_{6}$, $d(e_{2}) = d_{11}e_{2} + d_{11}e_{3} + d_{62}e_{6}$, $d(e_{3}) = d_{11}e_{2} + d_{11}e_{3} - d_{62}e_{6}$, $d(e_{4}) = d_{11}e_{4} + d_{64}e_{6}$, $d(e_{5}) = -\gamma d_{51}e_{1} + d_{11}e_{5} + d_{65}e_{6}$, $d(e_{6}) = 2d_{11}e_{6}$} & $K^{2}\times K^4$ \\\hline
\end{longtable}
\end{theo}

\begin{proof} On consid\`{e}re les alg\`{e}bres $N_{4, 6}$ ; $N_{5, 10}(\alpha)$ \`{a} $N_{5, 12}(\alpha, \beta)$ et $N_{6, 19}$ \`{a} $N_{6, 24}(\alpha, \beta, \gamma)$. La table de multiplication de la base naturelle de ces alg\`{e}bres v\'{e}rifient la r\'{e}lation $e_{1}^{2} = e_{2} + e_{3}$, $e_{2}^2 = e_{n}$, $e_{3}^{2} = -e_{n}$ et $e_{n}^{2} = 0$ o\`{u} $n=\dim(N) \geq 4$. Ainsi, $(a)$ entra\^{\i}ne $d_{12} = d_{21} = 0$, $d_{13} = d_{31} = 0$ et $(b)$ entra\^{\i}ne $d_{32} = d_{23}$. En d\'{e}rivant $e_{n} = e_{2}^{2}$, on obtient $d(e_{n}) = 2e_{2}d(e_{2}) = 2d_{22}e_{2}^{2} = 2d_{22}e_{n}$ ; on a aussi $d(e_{n}) = 2d_{33}e_{n}$ en d\'{e}rivant $e_{n} = -e_{3}^{2}$. Par cons\'{e}quent $d_{33}=d_{22}$.
En d\'{e}rivant $e_{2} + e_{3} = e_{1}^{2}$, on obtient $d(e_{2}) + d(e_{3}) = 2d_{11}d(e_{1}) = 2d_{11}e_{1}^{2} = 2d_{11}(e_{2} + e_{3})$. Ainsi, $\sum_{j = 2}^{n}(d_{j2} + d_{j3})e_{j} = 2d_{11}(e_{2} + e_{3})$  entra\^{\i}ne $d_{23} = 2d_{11} - d_{22}$ et $d_{j2} = -d_{j3}$ pour $j = 4, \ldots, n$. On en d\'eduit que
\begin{align}\label{Eq6}
d(e_{1}) &= d_{11}e_{1} + \sum_{j = 4}^{n}d_{j1}e_{j},\textrm{ }d(e_{2}) = d_{22}e_{2} + (2d_{11} - d_{22})e_{3} + \sum_{j = 4}^{n}d_{j2}e_{j},\notag\\
d(e_{3}) &= (2d_{11} - d_{22})e_{2} + d_{22}e_{3} - \sum_{j = 4}^{n}d_{j2}e_{j}\textrm{ et }d(e_{n}) = 2d_{22}e_{n}.
\end{align}

$\bullet$ $N_{4, 6}$ : $e_{1}^{2} = e_{2} + e_{3},$ $e_{2}^{2} = e_{4},$ $e_{3}^{2} = -e_{4},$ $e_{4}^{2} = 0$.
\eqref{Eq6} entraîne $d(e_{1}) = d_{11}e_{1} + d_{41}e_{4}$, $d(e_{2}) = d_{22}e_{2} + (2d_{11} - d_{22})e_{3} + d_{42}e_{4}$, $d(e_{3}) = (2d_{11} - d_{22})e_{2} + d_{22}e_{3} - d_{42}e_{4}$ et $d(e_{4}) = 2d_{22}e_{4}$.

$\bullet$ $N_{5, 10}(\alpha)$ : $e_{1}^{2} = e_{2} + e_{3},$ $e_{2}^{2} = e_{5},$ $e_{3}^{2} = -e_{5},$ $e_{4}^{2} =  \alpha e_{5},$ $e_{5}^{2} = 0$ avec $\alpha \in K^{*}$. De $(a)$ on a $d_{14} = d_{41} = 0$ ; $(b)$ donne $d_{24} = -\alpha d_{42}$ et $d_{34} = \alpha d_{43}$ ; de \eqref{Eq6} on a  $d(e_{1}) = d_{11}e_{1} + d_{51}e_{5}$, $d(e_{2}) = d_{22}e_{2} + (2d_{11} - d_{22})e_{3} + d_{42}e_{4} + d_{52}e_{5}$, $d(e_{3}) = (2d_{11} - d_{22})e_{2} + d_{22}e_{3} - d_{42}e_{4} - d_{52}e_{5}$, $d(e_{5}) = 2d_{22}e_{5}$. En d\'{e}rivant $\alpha e_{5} = e_{4}^{2}$, on obtient $\alpha d(e_{5}) = 2e_{4}d(e_{4}) = 2d_{44}e_{4}^{2} = 2\alpha d_{44}e_{5}$. Donc $d(e_{5}) = 2d_{44}e_{5}$ car $\alpha \neq 0$ et comme $d(e_{5}) = 2d_{22}e_{5}$, alors $d_{22} = d_{44}$. Puisque $d_{43} = -d_{42}$ et $d_{34} = \alpha d_{43}$ alors $d_{34} = -\alpha d_{42}$ ; ainsi $d(e_{4}) = -\alpha d_{42}e_{2} - \alpha d_{42}e_{3} + d_{22}e_{2} + d_{54}e_{5}$.

$\bullet$ $N_{5, 11}(\alpha)$ : $e_{1}^{2} = e_{2} + e_{3},$ $e_{2}^{2} = e_{5},$ $e_{3}^{2} = -e_{5},$ $e_{4}^{2} =  \alpha(e_{2} + e_{3}),$ $e_{5}^{2} = 0$ avec $\alpha \in K^{*}$. De $(a)$ on a $d_{24} = d_{42} = 0$ et $d_{34} = d_{43} = 0$ ; $(b)$ entra\^{\i}ne $d_{14} = -\alpha d_{41}$ ; de \eqref{Eq6} il vient $d(e_{1}) = d_{11}e_{1} + d_{41}e_{4} + d_{51}e_{5}$, $d(e_{2}) = d_{22}e_{2} + (2d_{11} - d_{22})e_{3} + d_{52}e_{5}$, $d(e_{3}) = (2d_{11} - d_{22})e_{2} + d_{22}e_{3} - d_{52}e_{5}$ et $d(e_{5}) = 2d_{22}e_{5}$. En d\'{e}rivant $e_{4}^{2} = \alpha e_{1}^{2}$, on obtient $2\alpha d_{44}e_{1}^{2} = 2\alpha d_{11}e_{1}^{2}$ ; donc $d_{44} = d_{11}$ car $2\alpha e_{1}^{2} \neq 0$. On en d\'{e}duit que $d(e_{4}) = -\alpha d_{41}e_{1} + d_{11}e_{4} + d_{54}e_{5}$.

$\bullet$ $N_{5, 12}(\alpha, \beta)$ : $e_{1}^{2} = e_{2} + e_{3},$ $e_{2}^{2} = e_{5},$ $e_{3}^{2} = -e_{5},$ $e_{4}^{2} =  \alpha(e_{2} + e_{3}) + \beta e_{5},$ $e_{5}^{2} = 0$ avec $\alpha, \beta \in K^{*}$. De $(a)$ on a $d_{14} = d_{41} = 0$, $d_{24} = d_{42} = 0$ et $d_{34} = d_{43} = 0$ ; \eqref{Eq6} entra\^{\i}ne $d(e_{1}) = d_{11}e_{1} + d_{51}e_{5}$, $d(e_{2}) = d_{22}e_{2} + (2d_{11} - d_{22})e_{3} + d_{52}e_{5}$, $d(e_{3}) = (2d_{11} - d_{22})e_{2} + d_{22}e_{3} - d_{52}e_{5}$ et $d(e_{5}) = 2d_{22}e_{5}$. En d\'{e}rivant $e_{4}^{2} = \alpha e_{1}^{2} + \beta e_{5}$, on obtient $2d_{44}e_{4}^{2} = 2\alpha d_{11}e_{1}^{2} + \beta d(e_{5})$, i.e. $2\alpha d_{44}(e_{2} + e_{3}) + 2\beta d_{44}e_{5} = 2\alpha d_{11}(e_{2} + e_{3}) + 2\beta d_{22} e_{5}$ ; donc $d_{11} = d_{44} = d_{22}$ car $\alpha\beta \neq 0$. En cons\'{e}quence  $d(e_{1}) = d_{11}e_{1} + d_{51}e_{1}$, $d(e_{2}) = d_{11}e_{2} + d_{11}e_{3} + d_{52}e_{5}$, $d(e_{3}) = d_{11}e_{2}  + d_{11}e_{3} - d_{52}e_{5}$, $d(e_{4}) = d_{11}e_{4} + d_{54}e_{5}$ et $d(e_{5}) = 2d_{11}e_{5}$.

$\bullet$ $N_{6, 19}(\alpha, \beta)$ : $e_{1}^{2} = e_{2} + e_{3}$, $e_{2}^{2} = e_{6}$, $e_{3}^{2} = -e_{6}$, $e_{4}^{2} = \alpha e_{6}$, $e_{5}^{2} = \beta e_{6}$, $e_{6}^{2} = 0$ avec $\alpha\beta \neq 0$. De $(a)$ on a $d_{14} = d_{41} = 0$ et $d_{15} = d_{51} = 0$ ; $(b)$ entra\^{\i}ne $d_{24} = -\alpha d_{42}$, $d_{25} = -\beta d_{52}$, $d_{34} = \alpha d_{43}$, $d_{35} = \beta d_{53}$ et $d_{45} = -\alpha^{-1}\beta d_{54}$ ; de \eqref{Eq6} on obtient $d(e_{1}) = d_{11}e_{1} + d_{61}e_{6}$, $d(e_{2}) = d_{22}e_{2} + (2d_{11} - d_{22})e_{3} + d_{42}e_{4} + d_{52}e_{5} + d_{62}e_{6}$, $d(e_{3}) = (2d_{11} - d_{22})e_{2} + d_{22}e_{3} - d_{42}e_{4} - d_{52}e_{5} - d_{62}e_{6}$ et $d(e_{6}) = 2d_{22}e_{6}$. En d\'{e}rivant $\alpha e_{6} = e_{4}^{2}$, on obtient $2\alpha d_{22}e_{6} = 2d_{44}e_{4}^{2} = 2\alpha d_{44}e_{6}$ et comme $d(e_{6}) = 2d_{22}e_{6}$, alors $d_{22} = d_{44}$. On montre de m\^eme que $d_{22} = d_{55}$ en d\'{e}rivant $\beta e_{6} = e_{5}^{2}$. Puisque $d_{43} = -d_{42}$ et $d_{53} = -d_{52}$, alors $d_{34} = \alpha d_{43} = -\alpha d_{42}$ et $d_{35} = \beta d_{53} = -\beta d_{52}$. On en d\'{e}duit que $d(e_{4}) = -\alpha d_{42}e_{2} - \alpha d_{42}e_{3} + d_{22}e_{4} + d_{54}e_{5} + d_{64}e_{6}$ et $d(e_{5}) = -\beta d_{52}e_{2} - \beta d_{52}e_{3} - \alpha^{-1}\beta d_{54}e_{4} + d_{22}e_{5} + d_{65}e_{6}$.

$\bullet$ $N_{6, 20}(\alpha, \beta)$ : $e_{1}^{2} = e_{2} + e_{3}$, $e_{2}^{2} = e_{6}$, $e_{3}^{2} = -e_{6}$, $e_{4}^{2} = \alpha(e_{2} + e_{3})$, $e_{5}^{2} = \beta e_{6}$, $e_{6}^{2} = 0$ avec $\alpha\beta \neq 0$. De $(a)$ on a $d_{15} = d_{51} = 0$, $d_{24} = d_{42} = 0$, $d_{34} = d_{43} = 0$ et $d_{45} = d_{54} = 0$ ; $(b)$ entra\^{\i}ne $d_{14} = -\alpha d_{41}$, $d_{25} = -\beta d_{52}$, $d_{35} = \beta d_{53}$ ; de \eqref{Eq6} on a $d(e_{1}) = d_{11}e_{1} + d_{41}e_{4} + d_{61}e_{6}$, $d(e_{2}) = d_{22}e_{2} + (2d_{11} - d_{22})e_{3} + d_{52}e_{5} + d_{62}e_{6}$, $d(e_{2}) = (2d_{11} - d_{22})e_{2} + d_{22}e_{3} - d_{52}e_{5} - d_{62}e_{6}$ et $d(e_{6}) = 2d_{22}e_{6}$. En d\'{e}rivant $e_{4}^{2} = \alpha e_{1}^{2}$, on obtient $2d_{44}e_{4}^{2} = 2\alpha d_{11}e_{1}^{2}$, i.e. $2\alpha d_{44}e_{1}^{2} = 2\alpha d_{11}e_{1}^{2}$ ; donc $d_{44} = d_{11}$. En d\'{e}rivant $e_{5}^{2} = \beta e_{6}$, on obtient $\beta d(e_{6}) = 2d_{55}e_{5}^{2}$, i.e. $2\beta d_{22}e_{6} = 2\beta d_{55}e_{6}$ ; donc $d_{22} = d_{55}$. Puisque $d_{53} = -d_{52}$, alors $d_{35} = \beta d_{53} = -\beta d_{52}$. Ainsi  $d(e_{4}) = -\alpha d_{41}e_{1} + d_{11}e_{4} + d_{64}e_{6}$ et $d(e_{5}) = -\beta d_{52} e_{2} - \beta d_{52} e_{3} + d_{22}e_{5} + d_{65}e_{6}$.

$\bullet$ $N_{6, 21}(\alpha, \beta, \gamma)$ : $e_{1}^{2} = e_{2} + e_{3}$, $e_{2}^{2} = e_{6}$, $e_{3}^{2} = -e_{6}$, $e_{4}^{2} = \alpha(e_{2} + e_{3}) + \beta e_{6}$, $e_{5}^{2} = \gamma e_{6}$, $e_{6}^{2} = 0$ avec $\alpha\beta\gamma \neq 0$. De $(a)$ on a $d_{14} = d_{41} = 0$, $d_{15} = d_{51} = 0$, $d_{24} = d_{42} = 0$, $d_{34} = d_{43} = 0$ et $d_{45} = d_{54} = 0$ ; $(b)$ entra\^{\i}ne $d_{25} = -\gamma d_{52}$, $d_{35} = \gamma d_{53}$ ; de \eqref{Eq6} on a $d(e_{1}) = d_{11}e_{1} + d_{61}e_{6}$, $d(e_{2}) = d_{22}e_{2} + (2d_{11} - d_{22})e_{3} + d_{52}e_{5} + d_{62}e_{6}$, $d(e_{3}) = (2d_{11} - d_{22})e_{2} + d_{22}e_{3} - d_{52}e_{5} - d_{62}e_{6}$ et $d(e_{6}) = 2d_{22}e_{6}$.
En d\'{e}rivant $e_{4}^{2} = \alpha e_{1}^{2} + \beta e_{6}$, on obtient $2d_{11}e_{1}^{2} + \beta d(e_{6}) = 2d_{44}e_{4}^{2} = 2\alpha d_{44}e_{1}^{2} + 2\beta d_{44}e_{6}$, i.e. $2d_{11}(e_{2} + e_{3}) + 2\beta d_{22}e_{6} = 2\alpha d_{44}(e_{2} + e_{3}) + 2\beta d_{44}e_{6}$. Donc $d_{11} = d_{44} = d_{22}$. En d\'{e}rivant $e_{5}^{2} = \gamma e_{6}$, on obtient $2\gamma d_{22}e_{6} = 2d_{55}e_{5}^{2} = 2\gamma d_{55}e_{6}$, donc $d_{22} = d_{55}$.
Puisque $d_{53} = -d_{52}$, alors $d_{35} = \gamma d_{53} = - \gamma d_{52}$. On en d\'{e}duit que $d(e_{2}) = d_{11}e_{2} + d_{11}e_{3} + d_{52}e_{5} + d_{62}e_{6}$, $d(e_{3}) = d_{11}e_{2} + d_{11}e_{3} - d_{52}e_{5} - d_{62}e_{6}$, $d(e_{4}) = d_{11}e_{4} + d_{64}e_{6}$, $d(e_{5}) = -\gamma d_{52} e_{2} - \gamma d_{52} e_{3} + d_{11}e_{5} + d_{65}e_{6}$, $d(e_{6}) = 2d_{11}e_{6}$.

$\bullet$ $N_{6, 22}(\alpha, \beta)$ : $e_{1}^{2} = e_{2} + e_{3}$, $e_{2}^{2} = e_{6}$, $e_{3}^{2} = -e_{6}$, $e_{4}^{2} = \alpha(e_{2} + e_{3})$, $e_{5}^{2} = \beta(e_{2} + e_{3})$, $e_{6}^{2} = 0$ avec $\alpha\beta \neq 0$. De $(a)$ on a $d_{24} = d_{42} = 0$, $d_{25} = d_{52} = 0$, $d_{34} = d_{43} = 0$ et $d_{35} = d_{53} = 0$ ; $(b)$ donne $d_{14} = -\alpha d_{41}$, $d_{15} = -\beta d_{51}$, $d_{45} = -\alpha^{-1}\beta d_{54}$ ; \eqref{Eq6} entra\^{\i}ne $d(e_{1}) = d_{11}e_{1} + d_{41}e_{4} + d_{51}e_{5} + d_{61}e_{6}$, $d(e_{2}) = d_{22}e_{2} + (2d_{11} - d_{22})e_{3} + d_{62}e_{6}$, $d(e_{3}) = (2d_{11} - d_{22})e_{2} + d_{22}e_{3} - d_{62}e_{6}$ et $d(e_{6}) = 2d_{22}e_{6}$. En d\'{e}rivant $e_{4}^{2} = \alpha e_{1}^{2}$, on obtient $2\alpha d_{11}e_{1}^{2} = 2d_{44}e_{4}^{2} = 2\alpha d_{44}e_{1}^{2}$ ; donc $d_{11} = d_{44}$. De fa\c{c}on analogue, en d\'{e}rivant $e_{5}^{2} = \beta e_{1}^{2}$, on obtient $d_{11} = d_{55}$. Il en r\'{e}sulte que $d(e_{4}) = -\alpha d_{41}e_{1} + d_{11}e_{4} + d_{54}e_{5} + d_{64}e_{6}$ et $d(e_{5}) = -\beta d_{51}e_{1} -\alpha^{-1}\beta d_{54}e_{4} + d_{11}e_{5} + d_{65}e_{6}$.

$\bullet$ $N_{6, 23}(\alpha, \beta, \gamma, \delta)$ : $e_{1}^{2} = e_{2} + e_{3}$, $e_{2}^{2} = e_{6}$, $e_{3}^{2} = -e_{6}$, $e_{4}^{2} = \alpha(e_{2} + e_{3}) + \beta e_{6}$, $e_{5}^{2} = \gamma(e_{2} + e_{3}) + \delta e_{6}$, $e_{6}^{2} = 0$ avec $\alpha\gamma \neq 0$, $\beta\delta \neq 0$, $\alpha\delta - \beta\gamma \neq 0$. De $(a)$ on a $d_{14} = d_{41} = 0$, $d_{15} = d_{51} = 0$, $d_{24} = d_{42} = 0$, $d_{25} = d_{52} = 0$, $d_{34} = d_{43} = 0$, $d_{35} = d_{53} = 0$ et $d_{45} = d_{54} = 0$ ; \eqref{Eq6} donne $d(e_{1}) = d_{11}e_{1} + d_{61}e_{6}$, $d(e_{2}) = d_{22}e_{2} + (2d_{11} - d_{22})e_{3} + d_{62}e_{6}$, $d(e_{3}) = (2d_{11} - d_{22})e_{2} + d_{22}e_{3} - d_{62}e_{6}$ et $d(e_{6}) = 2d_{22}e_{6}$. En d\'{e}rivant $e_{4}^{2} = \alpha e_{1}^{2} + \beta e_{6}$, on obtient $2\alpha d_{11}(e_{2} + e_{3}) + 2\beta d_{22}e_{6} = 2d_{44}e_{4}^{2} = 2\alpha d_{44}(e_{2} + e_{3}) + 2\beta d_{44}e_{6}$ ; donc $d_{11} = d_{44} = d_{22}$. De fa\c{c}on analogue, en d\'{e}rivant $e_{5}^{2} = \gamma e_{1}^{2} + \delta e_{6}$, on obtient $d_{11} = d_{55} = d_{22}$. On en d\'{e}duit que $d(e_{2}) = d_{11}e_{2} + d_{11}e_{3} + d_{62}e_{6}$, $d(e_{3}) = d_{11}e_{2} + d_{11}e_{3} - d_{62}e_{6}$, $d(e_{4}) = d_{11}e_{4} + d_{64}e_{6}$, $d(e_{5}) = d_{11}e_{5} + d_{65}e_{6}$ et $d(e_{6}) = 2d_{11}e_{6}$.

$\bullet$ $N_{6, 24}(\alpha, \beta, \gamma)$ : $e_{1}^{2} = e_{2} + e_{3}$, $e_{2}^{2} = e_{6}$, $e_{3}^{2} = -e_{6}$, $e_{4}^{2} = \alpha(e_{2} + e_{3}) + \beta e_{6}$, $e_{5}^{2} = \gamma(e_{2} + e_{3})$, $e_{6}^{2} = 0$ avec $\alpha\beta\gamma \neq 0$. De $(a)$ on a $d_{14} = d_{41} = 0$, $d_{24} = d_{42} = 0$, $d_{25} = d_{52} = 0$, $d_{34} = d_{43} = 0$, $d_{35} = d_{53} = 0$ et $d_{45} = d_{54} = 0$ ; $(b)$ donne $d_{15} = -\gamma d_{51}$ ; \eqref{Eq6} entra\^{\i}ne $d(e_{1}) = d_{11}e_{1} + d_{51}e_{5} + d_{61}e_{6}$, $d(e_{2}) = d_{22}e_{2} + (2d_{11} - d_{22})e_{3} + d_{62}e_{6}$, $d(e_{3}) = (2d_{11} - d_{22})e_{2} + d_{22}e_{3} - d_{62}e_{6}$ et $d(e_{6}) = 2d_{22}e_{6}$. En d\'{e}rivant $e_{4}^{2} = \alpha e_{1}^{2} + \beta e_{6}$, on obtient $2\alpha d_{11}(e_{2} + e_{3}) + 2\beta d_{22}e_{6} = 2d_{44}e_{4}^{2} = 2\alpha d_{44}(e_{2} + e_{3}) + 2\beta d_{44}e_{6}$ ; donc $d_{11} = d_{44} = d_{22}$. En d\'{e}rivant $e_{5}^{2} = \gamma e_{1}^{2}$, on obtient $2\gamma d_{11}(e_{2} + e_{3}) = 2d_{55}e_{5}^{2} = 2\gamma d_{55}(e_{2} + e_{3})$ ; donc $d_{11} = d_{55}$. Ainsi $d(e_{2}) = d_{11}e_{2} + d_{11}e_{3} + d_{62}e_{6}$, $d(e_{3}) = d_{11}e_{2} + d_{11}e_{3} - d_{62}e_{6}$, $d(e_{4}) = d_{11}e_{4} + d_{64}e_{6}$, $d(e_{5}) = -\gamma d_{51}e_{5} + d_{11}e_{5} + d_{65}e_{6}$ et $d(e_{6}) = 2d_{11}e_{6}$.
\end{proof}

\begin{prop} Soit $N$ une nil-alg\`{e}bre d'\'{e}volution ind\'{e}composable \`{a} puissances associatives, de dimension $\leq 6$, qui n'est pas associatives, telle que $\dim(ann(N)) =~1$. Alors le crochet des d\'{e}rivations $d$ et $d'$ de $N$ est donn\'{e} par la Table~$8$.

\begin{longtable}[c]{|p{2.5cm}|p{9cm}|c|}
 \caption*{\underline{\textbf{Table 8}} }\\
  \hline
    {$N$} &   {\small D\'{e}finition de la d\'{e}rivation $[d, d']$} &   {\small $\D(N)'\simeq$}   \\\hline
\endfirsthead
 \hline   {$N$} &   {\small D\'{e}finition de la d\'{e}rivation $[d, d']$} &   {\small $\D(N)'\simeq$}   \\\hline
\endhead \hline
{$N_{4,6}$} & {$[d, d']_{11} = [d, d']_{22} = 0$, $[d, d']_{41} =  (d_{11}'d_{41} - d_{11}d_{41}') + 2(d_{41}'d_{22} - d_{41}d_{22}')$, $[d, d']_{42} = 2(d_{42}'d_{11} - d_{42}d_{11}')$} & {$K^{2}$}   \\\hline
{$N_{5,10}(\alpha)$} & {$[d, d']_{11} = [d, d']_{22} = 0$, $[d, d']_{51} = (d_{11}'d_{51} - d_{11}d_{51}') +  2(d_{51}'d_{22} - d_{51}d_{22}')$, $[d, d']_{42} = (d_{22}'d_{42} - d_{22}d_{42}') + 2(d_{42}'d_{11} - d_{42}d_{11}')$, $[d, d']_{52} = (d_{42}'d_{54} - d_{42}d_{54}') + 2(d_{52}'d_{11} - d_{52}d_{11}')$, $[d, d']_{54} = d_{54}'d_{22} - d_{54}d_{22}'$} & $K\times K^{3}$ \\\hline
{$N_{5,11}(\alpha)$} & {$[d, d']_{11} = [d, d']_{41} = [d, d']_{22} = 0$, $[d, d']_{51} = (d_{11}'d_{51} - d_{11}d_{51}') + (d_{41}'d_{54} - d_{41}d_{54}') + 2(d_{51}'d_{22} - d_{51}d_{22}')$, $[d, d']_{52} =  2(d_{52}'d_{11} - d_{52}'d_{11})$, $[d, d']_{54} = -\alpha(d_{41}'d_{51} - d_{41}d_{51}') + (d_{11}'d_{54} - d_{11}d_{54}') + 2(d_{54}'d_{22} - d_{54}d_{22}')$} & $K^{3}$ \\\hline
{$N_{5,12}(\alpha, \beta)$} & {$[d, d']_{11} = 0$, $[d, d']_{51} =  d_{51}'d_{11} - d_{51}d_{11}'$, $[d, d']_{52} =  2(d_{52}'d_{11} - d_{52}d_{11}')$, $[d, d']_{54} = d_{54}'d_{11} - d_{54}d_{11}'$} & $K^{3}$ \\\hline
{$N_{6,19}(\alpha, \beta)$} & {$[d, d']_{11} = [d, d']_{22} = [d, d']_{54} = 0$, $[d, d']_{61} = (d_{11}'d_{61} - d_{11}d_{61}') + 2(d_{61}'d_{22} - d_{61}d_{22}')$, $[d, d']_{42} = (d_{22}'d_{42} - d_{22}d_{42}') + 2(d_{11}d_{42}' - d_{11}'d_{42}) - \alpha^{-1}\beta(d_{52}'d_{54} - d_{52}d_{54}')$, $[d, d']_{52} = (d_{42}'d_{54} - d_{42}d_{54}') + (d_{22}'d_{52} - d_{22}d_{52}') + 2(d_{11}d_{52}' - d_{11}'d_{52})$, $[d, d']_{62} = (d_{42}'d_{64} - d_{42}d_{64}') + (d_{52}'d_{65} - d_{52}d_{65}') + 2(d_{62}'d_{11} - d_{62}d_{11}')$, $[d, d']_{64} = (d_{64}'d_{22} - d_{64}d_{22}') + (d_{54}'d_{65} - d_{54}d_{65}')$,
$[d, d']_{65} = -\alpha^{-1}\beta(d_{54}'d_{64} - d_{54}d_{64}') + (d_{65}'d_{22} - d_{65}d_{22}')$} & $K^2\times K^{4}$ \\\hline
{$N_{6,20}(\alpha, \beta)$} & {$[d, d']_{11} = [d, d']_{41} = [d, d']_{22}  = 0$, $[d, d']_{61} = (d_{11}'d_{61} - d_{11}d_{61}') + (d_{41}'d_{64} - d_{41}d_{64}') + 2(d_{61}'d_{22} - d_{61}d_{22}')$, $[d, d']_{52} = (d_{22}'d_{52} - d_{22}d_{52}') + 2(d_{11}d_{52}' - d_{11}'d_{52})$, $[d, d']_{62} = (d_{52}'d_{65} - d_{52}d_{65}') + 2(d_{11}d_{62}' - d_{11}'d_{62})$, $[d, d']_{64} = -\alpha(d_{41}'d_{61} - d_{41}d_{61}') + (d_{11}'d_{64} - d_{11}d_{64}') + 2(d_{64}'d_{22} - d_{64}d_{22}')$, $[d, d']_{65} = d_{65}'d_{22} - d_{65}d_{22}'$} & $K\times K^{4}$ \\\hline
{$N_{6,21}(\alpha, \beta, \gamma)$} & {$[d, d']_{11} = 0$, $[d, d']_{61} = d_{61}'d_{11} - d_{61}d_{11}'$, $[d, d']_{52} = d_{52}'d_{11} - d_{52}d_{11}'$, $[d, d']_{62} = (d_{52}'d_{65} - d_{52}d_{65}')  + 2(d_{62}'d_{11} - d_{62}'d_{11})$, $[d, d']_{64} = d_{64}'d_{11} - d_{64}'d_{11}$, $[d, d']_{65} = d_{65}'d_{11} - d_{65}d_{11}'$} & $K\times K^{4}$ \\\hline
{$N_{6,22}(\alpha, \beta)$} & {$[d, d']_{11} = [d, d']_{22} = 0$, $[d, d']_{41} = \alpha^{-1}\beta(d_{51}d_{54}' - d_{51}'d_{54})$, $[d, d']_{51} =  d_{41}'d_{54} - d_{41}d_{54}'$, $[d, d']_{61} = (d_{11}'d_{61} - d_{11}d_{61}') + (d_{41}'d_{64} - d_{41}d_{64}') + (d_{51}'d_{65} - d_{51}d_{65}') + 2(d_{61}'d_{22} - d_{61}d_{22}')$, $[d, d']_{62} =  2(d_{11}d_{62}' - d_{11}'d_{62})$, $[d, d']_{54} = \alpha(d_{41}d_{51}' - d_{41}'d_{51})$, $[d, d']_{64} = (d_{64}'d_{11} - d_{64}d_{11}') + (d_{54}'d_{65} - d_{54}d_{65}') + \alpha(d_{41}d_{61}' - d_{41}'d_{61})$, $[d, d']_{65} = -\beta(d_{51}'d_{61} - d_{51}d_{61}') - \alpha^{-1}\beta(d_{54}'d_{64} - d_{54}d_{64}') + (d_{11}'d_{65} - d_{11}d_{65}') + 2(d_{65}'d_{22} - d_{65}d_{22}')$} & $K^3\times K^{4}$ \\\hline
{$N_{6,23}(\alpha, \beta, \gamma, \delta)$} & {$[d, d']_{11} = 0$, $[d, d']_{61} = d_{61}'d_{11} - d_{11}'d_{61}$, $[d, d']_{62} = 2(d_{62}'d_{11} - d_{62}d_{11}')$, $[d, d']_{64} = d_{64}'d_{11} - d_{64}d_{11}'$, $[d, d']_{65} = d_{65}'d_{11} - d_{65}d_{11}'$} & $K^{4}$ \\\hline
{$N_{6,24}(\alpha, \beta, \gamma)$} & {$[d, d']_{11} = [d, d']_{51} = 0$, $[d, d']_{61} = (d_{61}'d_{11} - d_{61}d_{11}') + (d_{51}'d_{65} - d_{51}d_{65}')$, $[d, d']_{62} = 2(d_{62}'d_{11} - d_{62}d_{11}')$, $[d, d']_{64} = d_{64}'d_{11} - d_{64}d_{11}'$, $[d, d']_{65} = \gamma(d_{51}d_{61}' - d_{51}'d_{61}) + (d_{65}'d_{11} - d_{65}d_{11}')$} & $K^{4}$ \\\hline
\end{longtable}
\end{prop}

\begin{proof} Soient $d, d' \in \D(N)$. Calculons $[d,d']$.

$\bullet$ $N_{4, 6}$: On a $d\circ d' (e_{1}) = d_{11}'d(e_{1}) + d_{41}'d(e_{4})$, alors $[d, d']_{11} = d_{11}'d_{11} - d_{11}d_{11}' = 0$, $[d, d']_{41} = (d_{11}'d_{41} + 2d_{41}'d_{22}) - (d_{11}d_{41}' + 2d_{41}d_{22}') = (d_{11}'d_{41} - d_{11}d_{41}') + 2(d_{41}'d_{22} - d_{41}d_{22}')$. On a $d\circ d' (e_{2}) = d_{22}'d(e_{2}) + (2d_{11}' - d_{22}')d(e_{3}) + d_{42}'d(e_{4})$, alors $[d, d']_{22} = (d_{22}'d_{22} + (2d_{11}' - d_{22}')(2d_{11} - d_{22})) - (d_{22}d_{22}' + (2d_{11} - d_{22})(2d_{11}' - d_{22}')) = 0$, $[d, d']_{42} = (d_{22}'d_{42} - (2d_{11}' - d_{22}')d_{42} + 2d_{42}'d_{22}) - (d_{22}d_{42}' - (2d_{11} - d_{22})d_{42}' + 2d_{42}d_{22}') = 2(d_{11}d_{42}' - d_{11}'d_{42})$. On en d\'{e}duit que $\D(N_{4, 6})'\simeq K^{2}$.

$\bullet$ $N_{5, 10}(\alpha)$: On a $d\circ d' (e_{1}) = d_{11}'d(e_{1}) + d_{51}'d(e_{5})$, alors $[d, d']_{11} = d_{11}'d_{11} - d_{11}d_{11}' = 0$, $[d, d']_{51} = (d_{11}'d_{51} + 2d_{51}'d_{22}) - (d_{11}d_{51}' + 2d_{51}d_{22}') = (d_{11}'d_{51} - d_{11}d_{51}') +  2(d_{51}'d_{22} - d_{51}d_{22}')$.  On a $d\circ d' (e_{2}) = d_{22}'d(e_{2}) + (2d_{11}' - d_{22}')d(e_{3}) + d_{42}'d(e_{4}) + d_{52}'d(e_{5})$, alors $[d, d']_{22} = (d_{22}'d_{22} + (2d_{11}' - d_{22}')(2d_{11} - d_{22}) -  \alpha d_{42}'d_{42}) - (d_{22}d_{22}' + (2d_{11} - d_{22})(2d_{11}' - d_{22}') -  \alpha d_{42}d_{42}') = 0$, $[d, d']_{42} = (d_{22}'d_{42} - (2d_{11}' - d_{22}')d_{42} + d_{42}'d_{22}) - (d_{22}d_{42}' - (2d_{11} - d_{22})d_{42}' + d_{42}d_{22}') = (d_{22}'d_{42} - d_{22}d_{42}') + 2(d_{42}'d_{11} - d_{42}d_{11}')$, $[d, d']_{52} = (d_{22}'d_{52} - (2d_{11}' - d_{22}')d_{52} + d_{42}'d_{54} + 2d_{52}'d_{22}) - (d_{22}d_{52}' - (2d_{11} - d_{22})d_{52}' + d_{42}d_{54}' + 2d_{52}d_{22}') = (d_{42}'d_{54} - d_{42}d_{54}') + 2(d_{52}'d_{11} - d_{52}d_{11}')$. On a $d\circ d' (e_{4}) = -\alpha d_{42}'d(e_{2}) - \alpha d_{42}'d(e_{3}) + d_{22}'d(e_{4}) + d_{54}'d(e_{5})$, alors $[d, d']_{54} = (-\alpha d_{42}'d_{52} + \alpha d_{42}'d_{52} + d_{22}'d_{54} + 2d_{54}'d_{22}) - (-\alpha d_{42}d_{52}' + \alpha d_{42}d_{52}' + d_{22}d_{54}' + 2d_{54}d_{22}') =  d_{54}'d_{22} - d_{54}d_{22}'$. On en d\'{e}duit que $\D(N_{5, 10}(\alpha))'\simeq K^{4}$.

$\bullet$ $N_{5, 11}(\alpha))$: On a $d\circ d' (e_{1}) = d_{11}'d(e_{1}) + d_{41}'d(e_{4}) + d_{51}'d(e_{5})$, alors $[d, d']_{11} = (d_{11}'d_{11} - \alpha d_{41}'d_{41}) - (d_{11}d_{11}' - \alpha d_{41}d_{41}') = 0$, $[d, d']_{41} = (d_{11}'d_{41} + d_{41}'d_{11}) - (d_{11}d_{41}' + d_{41}d_{11}') = 0$, $[d, d']_{51} = (d_{11}'d_{51} + d_{41}'d_{54} + 2d_{51}'d_{22}) - (d_{11}d_{51}' + d_{41}d_{54}' + 2d_{51}d_{22}') = (d_{11}'d_{51} - d_{11}d_{51}') + (d_{41}'d_{54} - d_{41}d_{54}') + 2(d_{51}'d_{22} - d_{51}d_{22}')$.  On a $d\circ d' (e_{2}) = d_{22}'d(e_{2}) + (2d_{11}' - d_{22}')d(e_{3}) + d_{52}'d(e_{5})$, alors $[d, d']_{22} = (d_{22}'d_{22} + (2d_{11}' - d_{22}')(2d_{11} - d_{22})) - (d_{22}d_{22}' + (2d_{11} - d_{22})(2d_{11}' - d_{22}')) = 0$, $[d, d']_{52} = (d_{22}'d_{52} - (2d_{11}' - d_{22}')d_{52} + 2d_{52}'d_{22}) - (d_{22}d_{52}' - (2d_{11} - d_{22})d_{52}' + 2d_{52}d_{22}') = 2(d_{52}'d_{11} - d_{52}'d_{11})$. On a $d\circ d' (e_{4}) = -\alpha d_{41}'d(e_{1}) + d_{11}'d(e_{4}) + d_{54}'d(e_{5})$, alors $[d, d']_{54} = (-\alpha d_{41}'d_{51} + d_{11}'d_{54} + 2d_{54}'d_{22}) - (-\alpha d_{41}d_{51}' + d_{11}d_{54}' + 2d_{54}d_{22}') = -\alpha(d_{41}'d_{51} - d_{41}d_{51}') + (d_{11}'d_{54} - d_{11}d_{54}') + 2(d_{54}'d_{22}) - d_{54}d_{22}')$. On en d\'{e}duit que $\D(N_{5, 11}(\alpha))'\simeq K^{3}$.

$\bullet$ $N_{5, 12}(\alpha, \beta)$: On a $d\circ d' (e_{1}) = d_{11}'d(e_{1}) + d_{51}'d(e_{5})$, alors $[d, d']_{11} = d_{11}'d_{11} - d_{11}d_{11}' = 0$, $[d, d']_{51} = (d_{11}'d_{51} + 2d_{51}'d_{11}) - (d_{11}d_{51}' + 2d_{51}d_{11}') = d_{51}'d_{11} - d_{51}d_{11}'$. On a $d\circ d' (e_{2}) = d_{11}'d(e_{2}) + d_{11}'d(e_{3}) + d_{52}'d(e_{5})$, alors $[d, d']_{52} = (d_{11}'d_{52} - d_{11}'d_{52} + 2d_{52}'d_{11}) - (d_{11}d_{52}' - d_{11}d_{52}' + 2d_{52}d_{11}') = 2(d_{52}'d_{11} - d_{52}d_{11}')$. On a $d\circ d' (e_{4}) = d_{11}'d(e_{4}) + d_{54}'d(e_{5})$, alors $[d, d']_{54} = (d_{11}'d_{54} + 2d_{54}'d_{11}) - (d_{11}d_{54}' + 2d_{54}d_{11}') = d_{54}'d_{11}) - d_{54}d_{11}'$.
On en d\'{e}duit que $\D(N_{5, 12}(\alpha, \beta))'\simeq K^{3}$.

$\bullet$ $N_{6, 19}(\alpha, \beta))$:
On a $d\circ d' (e_{1}) = d_{11}'d(e_{1}) + d_{61}'d(e_{6})$, alors $[d, d']_{11} = d_{11}'d_{11} - d_{11}d_{11}' =  0$, $[d, d']_{61} = (d_{11}'d_{61} + 2d_{61}'d_{22}) - (d_{11}d_{61}' + 2d_{61}d_{22}') = (d_{11}'d_{61} - d_{11}d_{61}') + 2(d_{61}'d_{22} - d_{61}d_{22}')$.
On a $d\circ d' (e_{2}) = d_{22}'d(e_{2}) + (2d_{11}' - d_{22}')d(e_{3}) + d_{42}'d(e_{4}) + d_{52}'d(e_{5}) + d_{62}'d(e_{6})$, alors $[d, d']_{22} = (d_{22}'d_{22} + (2d_{11}' - d_{22}')(2d_{11} - d_{22}) - \alpha d_{42}'d_{42} - \beta d_{52}'d_{52}) - (d_{22}d_{22}' + (2d_{11} - d_{22})(2d_{11}' - d_{22}') - \alpha d_{42}d_{42}' - \beta d_{52}d_{52}') = 0$, $[d, d']_{42} = (d_{22}'d_{42} - (2d_{11}' - d_{22}')d_{42} + d_{42}'d_{22} - \alpha^{-1}\beta d_{52}'d_{54}) - (d_{22}d_{42}' - (2d_{11} - d_{22})d_{42}' + d_{42}d_{22}' - \alpha^{-1}\beta d_{52}d_{54}') = (d_{22}'d_{42} - d_{22}d_{42}') + 2(d_{11}d_{42}' - d_{11}'d_{42}) - \alpha^{-1}\beta(d_{52}'d_{54} - d_{52}d_{54}')$, $[d, d']_{52} = (d_{22}'d_{52} - (2d_{11}' - d_{22}')d_{52} + d_{42}'d_{54} + d_{52}'d_{22}) - (d_{22}d_{52}' - (2d_{11} - d_{22})d_{52}' + d_{42}d_{54}' + d_{52}d_{22}') = (d_{42}'d_{54} - d_{42}d_{54}') + (d_{22}'d_{52} - d_{22}d_{52}') + 2(d_{11}d_{52}' - d_{11}'d_{52})$, $[d, d']_{62} = (d_{22}'d_{62} - (2d_{11}' - d_{22}')d_{62} + d_{42}'d_{64} + d_{52}'d_{65} + 2d_{62}'d_{22}) - (d_{22}d_{62}' - (2d_{11} - d_{22})d_{62}' + d_{42}d_{64}' + d_{52}d_{65}' + 2d_{62}d_{22}') = (d_{42}'d_{64} - d_{42}d_{64}') + (d_{52}'d_{65} - d_{52}d_{65}') + 2(d_{62}'d_{11} - d_{62}d_{11}')$. On a $d\circ d' (e_{4}) = -\alpha d_{42}'d(e_{2}) - \alpha d_{42}'d(e_{3}) + d_{22}'d(e_{4}) + d_{54}'d(e_{5}) + d_{64}'d(e_{6})$, alors $[d, d']_{54} = (-\alpha d_{42}'d_{52} + \alpha d_{42}'d_{52} + d_{22}'d_{54} + d_{54}'d_{22}) - (-\alpha d_{42}d_{52}' + \alpha d_{42}d_{52}' + d_{22}d_{54}' + d_{54}d_{22}') = 0$, $[d, d']_{64} = (-\alpha d_{42}'d_{62} + \alpha d_{42}'d_{62} + d_{22}'d_{64} + d_{54}'d_{65} + 2d_{64}'d_{22}) - (-\alpha d_{42}d_{62}' + \alpha d_{42}d_{62}' + d_{22}d_{64}' + d_{54}d_{65}' + 2d_{64}d_{22}') = (d_{64}'d_{22} - d_{64}d_{22}') + (d_{54}'d_{65} - d_{54}d_{65}')$. On a  $d\circ d' (e_{5}) = -\beta d_{52}'d(e_{2}) - \beta d_{52}'d(e_{3}) - \alpha^{-1}\beta d_{54}'d(e_{4}) + d_{22}'d(e_{5}) + d_{65}'d(e_{6})$, alors $[d, d']_{65} = (-\beta d_{52}'d_{62} + \beta d_{52}'d_{62} - \alpha^{-1}\beta d_{54}'d_{64} + d_{22}'d_{65} + 2d_{65}'d_{22}) - (-\beta d_{52}d_{62}' + \beta d_{52}d_{62}' - \alpha^{-1}\beta d_{54}d_{64}' + d_{22}d_{65}' + 2d_{65}d_{22}') = -\alpha^{-1}\beta(d_{54}'d_{64} - d_{54}d_{64}') + (d_{65}'d_{22} - d_{65}d_{22}')$. On en d\'{e}duit que $\D(N_{6, 19}(\alpha, \beta))' \simeq K^{6}$.

$\bullet$ $N_{6, 20}(\alpha, \beta)$: On a $d\circ d' (e_{1}) = d_{11}'d(e_{1}) + d_{41}'d(e_{4}) + d_{61}'d(e_{6})$, alors $[d, d']_{11} = (d_{11}'d_{11} - \alpha d_{41}'d_{41}) - (d_{11}d_{11}' - \alpha d_{41}d_{41}') = 0$, $[d, d']_{41} = (d_{11}'d_{41} + d_{41}'d_{11}) - (d_{11}d_{41}' + d_{41}d_{11}') = 0$, $[d, d']_{61} = (d_{11}'d_{61} + d_{41}'d_{64} + 2d_{61}'d_{22}) - (d_{11}d_{61}' + d_{41}d_{64}' + 2d_{61}d_{22}') = (d_{11}'d_{61} - d_{11}d_{61}') + (d_{41}'d_{64} - d_{41}d_{64}') + 2(d_{61}'d_{22} - d_{61}d_{22}')$. On a $d\circ d' (e_{2}) = d_{22}'d(e_{2}) + (2d_{11}' - d_{22}')d(e_{3}) + d_{52}'d(e_{5}) + d_{62}'d(e_{6})$, alors $[d, d']_{22} = (d_{22}'d_{22} + (2d_{11}' - d_{22}')(2d_{11} - d_{22}) - \beta d_{52}'d_{52}) - (d_{22}d_{22}' + (2d_{11} - d_{22})(2d_{11}' - d_{22}') - \beta d_{52}d_{52}') = 0$, $[d, d']_{52} = (d_{22}'d_{52} - (2d_{11}' - d_{22}')d_{52} + d_{52}'d_{22}) - (d_{22}d_{52}' - (2d_{11} - d_{22})d_{52}' + d_{52}d_{22}') = (d_{22}'d_{52} - d_{22}d_{52}') + 2(d_{11}d_{52}' - d_{11}'d_{52})$, $[d, d']_{62} = (d_{22}'d_{62} - (2d_{11}' - d_{22}')d_{62} + d_{52}'d_{65} + 2d_{62}'d_{22}) - (d_{22}d_{62}' - (2d_{11} - d_{22})d_{62}' + d_{52}d_{65}' + 2d_{62}d_{22}') = (d_{52}'d_{65} - d_{52}d_{65}') + 2(d_{11}d_{62}' - d_{11}'d_{62})$. On a ; $d\circ d' (e_{4}) = -\alpha d_{41}'d(e_{1}) + d_{11}'d(e_{4}) + d_{64}'d(e_{6})$, alors $[d, d']_{64} = (-\alpha d_{41}'d_{61} + d_{11}'d_{64} + 2d_{64}'d_{22}) - (-\alpha d_{41}d_{61}' + d_{11}d_{64}' + 2d_{64}d_{22}') = -\alpha(d_{41}'d_{61} - d_{41}d_{61}') + (d_{11}'d_{64} - d_{11}d_{64}') + 2(d_{64}'d_{22} - d_{64}d_{22}')$. On a $d\circ d' (e_{5}) = -\beta d_{52}'d(e_{2}) - \beta d_{52}'d(e_{3}) + d_{22}'d(e_{5}) + d_{65}'d(e_{6})$, alors $[d, d']_{65} = (-\beta d_{52}'d_{62} + \beta d_{52}'d_{62} + d_{22}'d_{65} + 2d_{65}'d_{22}) - (-\beta d_{52}d_{62}' + \beta d_{52}d_{62}' + d_{22}d_{65}' + 2d_{65}d_{22}') = d_{65}'d_{22} - d_{65}d_{22}'$. On en d\'{e}duit que $\D(N_{6, 20}(\alpha, \beta))'\simeq K^{5}$.

$\bullet$ $N_{6, 21}(\alpha, \beta, \gamma)$: On a $d\circ d' (e_{1}) = d_{11}'d(e_{1}) + d_{61}'d(e_{6})$, alors $[d, d']_{11} = (d_{11}'d_{11} - d_{11}'d_{11}) = 0$, $[d, d']_{61} = (d_{11}'d_{61} + 2d_{61}'d_{11}) - (d_{11}d_{61}' +  2d_{61}d_{11}') = d_{61}'d_{11} - d_{61}d_{11}'$. On a $d\circ d' (e_{2}) = d_{11}'d(e_{2}) + d_{11}'d(e_{3}) + d_{52}'d(e_{5}) + d_{62}'d(e_{6})$, alors $[d, d']_{52} = (d_{11}'d_{52} - d_{11}'d_{52} + d_{52}'d_{11}) - (d_{11}d_{52}' - d_{11}d_{52}' + d_{52}d_{11}') = d_{52}'d_{11} - d_{52}d_{11}'$, $[d, d']_{62} = (d_{11}'d_{62} - d_{11}'d_{62} + d_{52}'d_{65} + 2d_{62}'d_{11}) - (d_{11}d_{62}' - d_{11}d_{62}' + d_{52}d_{65}' + 2d_{62}d_{11}') = (d_{52}'d_{65} - d_{52}d_{65}')  + 2(d_{62}'d_{11} - d_{62}'d_{11})$. On a $d\circ d' (e_{4}) = d_{11}'d(e_{4}) + d_{64}'d(e_{6})$, alors $[d, d']_{64} = (d_{11}'d_{64} + 2d_{64}'d_{11}) - (d_{11}d_{64}' + 2d_{64}d_{11}') = d_{64}'d_{11} - d_{64}'d_{11}$. On a $d\circ d' (e_{5}) = -\gamma d_{52}'d(e_{2}) - \gamma d_{52}'d(e_{3}) + d_{11}'d(e_{5}) + d_{65}'d(e_{6})$, alors $[d, d']_{65} = (-\gamma d_{52}'d_{62} + \gamma d_{52}'d_{62} + d_{11}'d_{65} + 2d_{65}'d_{11}) - (-\gamma d_{52}d_{62}' + \gamma d_{52}d_{62}' + d_{11}d_{65}' + 2d_{65}d_{11}') = d_{65}'d_{11} - d_{65}d_{11}'$. On en d\'{e}duit que $\D(N_{6, 21}(\alpha, \beta, \gamma))'\simeq K^{5}$.

$\bullet$ $N_{6, 22}(\alpha, \beta)$: On a $d\circ d' (e_{1}) = d_{11}'d(e_{1}) + d_{41}'d(e_{4}) + d_{51}'d(e_{5}) + d_{61}'d(e_{6})$, alors $[d, d']_{11} = (d_{11}'d_{11} - \alpha d_{41}'d_{41} - \beta d_{51}'d_{51}) - (d_{11}d_{11}' - \alpha d_{41}d_{41}' - \beta d_{51}d_{51}') = 0$, $[d, d']_{41} = (d_{11}'d_{41} + d_{41}'d_{11} - \alpha^{-1}\beta d_{51}'d_{54}) - (d_{11}d_{41}' + d_{41}d_{11}' - \alpha^{-1}\beta d_{51}d_{54}') = \alpha^{-1}\beta(d_{51}d_{54}' - d_{51}'d_{54})$, $[d, d']_{51} = (d_{11}'d_{51} + d_{41}'d_{54} + d_{51}'d_{11}) - (d_{11}d_{51}' + d_{41}d_{54}' + d_{51}d_{11}') = d_{41}'d_{54} - d_{41}d_{54}'$, $[d, d']_{61} = (d_{11}'d_{61} + d_{41}'d_{64} + d_{51}'d_{65} + 2d_{61}'d_{22}) - (d_{11}d_{61}' + d_{41}d_{64}' + d_{51}d_{65}' + 2d_{61}d_{22}') =  (d_{11}'d_{61} - d_{11}d_{61}') + (d_{41}'d_{64} - d_{41}d_{64}') + (d_{51}'d_{65} - d_{51}d_{65}') + 2(d_{61}'d_{22} - d_{61}d_{22}')$. On a $d\circ d' (e_{2}) = d_{22}'d(e_{2}) + (2d_{11}' - d_{22}')d(e_{3}) + d_{62}'d(e_{6})$, alors $[d, d']_{22} = (d_{22}'d_{22} + (2d_{11}' - d_{22}')(2d_{11} - d_{22})) - (d_{22}d_{22}' + (2d_{11} - d_{22})(2d_{11}' - d_{22}')) = 0$, $[d, d']_{62} = (d_{22}'d_{62} - (2d_{11}' - d_{22}')d_{62} + 2d_{62}'d_{22}) - (d_{22}d_{62}' - (2d_{11} - d_{22})d_{62}' + 2d_{62}d_{22}') = 2(d_{11}d_{62}' - d_{11}'d_{62})$. On a $d\circ d' (e_{4}) = -\alpha d_{41}'d(e_{1}) + d_{11}'d(e_{4}) + d_{54}'d(e_{5}) + d_{64}'d(e_{6})$, alors $[d, d']_{54} = (-\alpha d_{41}'d_{51} + d_{11}'d_{54} + d_{54}'d_{11}) - (-\alpha d_{41}d_{51}' + d_{11}d_{54}' + d_{54}d_{11}') = \alpha(d_{41}d_{51}' - d_{41}'d_{51})$, $[d, d']_{64} = (-\alpha d_{41}'d_{61} + d_{11}'d_{64} + d_{54}'d_{65} + 2d_{64}'d_{11}) - (-\alpha d_{41}d_{61}' + d_{11}d_{64}' + d_{54}d_{65}' + 2d_{64}d_{11}') = (d_{64}'d_{11} - d_{64}d_{11}') + (d_{54}'d_{65} - d_{54}d_{65}') + \alpha(d_{41}d_{61}' - d_{41}'d_{61})$. On a $d\circ d' (e_{5}) = -\beta d_{51}'d(e_{1}) - \alpha^{-1}\beta d_{54}'d(e_{4}) + d_{11}'d(e_{5}) + d_{65}'d(e_{6})$, alors $[d, d']_{65} = (-\beta d_{51}'d_{61} - \alpha^{-1}\beta d_{54}'d_{64} + d_{11}'d_{65} + 2d_{65}'d_{22}) - (-\beta d_{51}d_{61}' - \alpha^{-1}\beta d_{54}d_{64}' + d_{11}d_{65}' + 2d_{65}d_{22}') = -\beta(d_{51}'d_{61} - d_{51}d_{61}') - \alpha^{-1}\beta(d_{54}'d_{64} - d_{54}d_{64}') + (d_{11}'d_{65} - d_{11}d_{65}') + 2(d_{65}'d_{22} - d_{65}d_{22}')$. On en d\'{e}duit que $\D(N_{6, 22}(\alpha, \beta))' \simeq K^{7}$.

$\bullet$ $N_{6, 23}(\alpha, \beta)$: On a $d\circ d' (e_{1}) = d_{11}'d(e_{1}) + d_{61}'d(e_{6})$, alors $[d, d']_{11} = d_{11}'d_{11} - d_{11}d_{11}' = 0$, $[d, d']_{61} = (d_{11}'d_{61} + 2d_{61}'d_{11}) - (d_{11}d_{61}' + 2d_{61}d_{11}') = d_{61}'d_{11} - d_{11}'d_{61}$. On a $d\circ d' (e_{2}) = d_{11}'d(e_{2}) + d_{11}'d(e_{3}) + d_{62}'d(e_{6})$, alors $[d, d']_{62} = (d_{11}'d_{62} - d_{11}'d_{62} + 2d_{62}'d_{11}) - (d_{11}d_{62}' - d_{11}d_{62}' + 2d_{62}d_{11}') = 2(d_{62}'d_{11} - d_{62}d_{11}')$. On a $d\circ d' (e_{4}) = d_{11}'d(e_{4}) + d_{64}'d(e_{6})$, alors $[d, d']_{64} = (d_{11}'d_{64} + 2d_{64}'d_{11}) - (d_{11}d_{64}' + 2d_{64}d_{11}') = d_{64}'d_{11} - d_{64}d_{11}'$.
On a  $d\circ d' (e_{5}) = d_{11}'d(e_{5}) + d_{65}'d(e_{6})$, alors $[d, d']_{65} = (d_{11}'d_{65} + 2d_{65}'d_{11}) - (d_{11}d_{65}' + 2d_{65}d_{11}') = d_{65}'d_{11} - d_{65}d_{11}'$. On en d\'{e}duit que $\D(N_{6, 23}(\alpha, \beta))'\simeq K^{4}$.

$\bullet$ $N_{6, 24}(\alpha, \beta, \gamma)$: On a $d\circ d' (e_{1}) = d_{11}'d(e_{1}) + d_{51}'d(e_{5}) + d_{61}'d(e_{6})$, alors $[d, d']_{11} = (d_{11}'d_{11} - \gamma d_{51}'d_{51}) -  (d_{11}d_{11}' - \gamma d_{51}d_{51}') = 0$, $[d, d']_{51} = (d_{11}'d_{51} + d_{51}'d_{11}) - (d_{11}d_{51}' + d_{51}d_{11}') = 0$, $[d, d']_{61} = (d_{11}'d_{61} + d_{51}'d_{65} + 2d_{61}'d_{11}) - (d_{11}d_{61}' + d_{51}d_{65}' + 2d_{61}d_{11}') = (d_{61}'d_{11} - d_{61}d_{11}') + (d_{51}'d_{65} - d_{51}d_{65}')$. On a $d\circ d' (e_{2}) = d_{11}'d(e_{2}) + d_{11}'d(e_{3}) + d_{62}'d(e_{6})$, alors $[d, d']_{62} = (d_{11}'d_{62} - d_{11}'d_{62} + 2d_{62}'d_{11}) - (d_{11}d_{62}' - d_{11}d_{62}' + 2d_{62}d_{11}') = 2(d_{62}'d_{11} - d_{62}d_{11}')$. On a $d\circ d' (e_{4}) = d_{11}'d(e_{4}) + d_{64}'d(e_{6})$, alors $[d, d']_{64} = (d_{11}'d_{64} + 2d_{64}'d_{11}) - (d_{11}d_{64}' + 2d_{64}d_{11}') = d_{64}'d_{11} - d_{64}d_{11}'$. On a $d\circ d' (e_{5}) = -\gamma d_{51}'d(e_{1}) + d_{11}'d(e_{5}) + d_{65}'d(e_{6})$, alors $[d, d']_{65} = (-\gamma d_{51}'d_{61} + d_{11}'d_{65} + 2d_{65}'d_{11}) - (-\gamma d_{51}d_{61}' + d_{11}d_{65}' + 2d_{65}d_{11}') = \gamma(d_{51}d_{61}' - d_{51}'d_{61}) + (d_{65}'d_{11} - d_{65}d_{11}')$. On en d\'{e}duit que $\D(N_{6, 24}(\alpha, \beta, \gamma))'\simeq K^{4}$.
\end{proof}


\begin{prop} Soit $N$ une nil-alg\`{e}bre d'\'{e}volution ind\'{e}composable \`{a} puissances associatives, de dimension $\leq 6$, qui n'est pas associative, telle que $\dim(ann(N)) = 1$. La Table~$9$, caract\'{e}rise les d\'{e}rivations int\'{e}rieures de $N$.

\begin{longtable}[c]{|p{2.5cm}|p{9cm}|c|}
 \caption*{\underline{\textbf{Table 9}} }\\
  \hline
    {$N$} &   {\small $R_{c}+[R_a,R_b] \in \D(N)$ si et seulement si} &   {\small $\In(N) \simeq$}   \\\hline
\endfirsthead
 \hline   {$N$} &   {\small $R_{c}+[R_a,R_b] \in \D(N)$ si et seulement si} &   {\small $\In(N) \simeq$}   \\\hline
\endhead \hline
{$N_{4,6}$} & {$R_c = c_{2}(e_{42}-e_{43})$}, $[R_a,R_b]=(a_2b_1-a_1b_2+a_1b_3-a_3b_1)e_{41}$ & {$K^2$}   \\\hline
{$N_{5,10}(\alpha)$} & {$R_c=c_2(e_{52}-e_{53})+\al c_4e_{54}$}, $[R_a,R_b]=(b_1(a_2-a_3)-a_1(b_2-b_3))e_{51}$ & $K^3$ \\\hline
{$N_{5,11}(\alpha)$} & {$R_c=c_2(e_{52}- e_{53})$}, $[R_a,R_b]=(b_1(a_2-a_3)-a_1(b_2-b_3))e_{51}+\al(b_4(a_2-a_3)-a_4(b_2-b_3))e_{54}$  & $K^{3}$ \\\hline
{$N_{5,12}(\alpha, \beta)$} & {$R_c=c_2(e_{52}-e_{53})$}, $[R_a,R_b]=(b_1(a_2-a_3)-a_1(b_2-b_3))e_{51}+\al(b_4(a_2-a_3)-a_4(b_2-b_3))e_{54}$ & $K^{3}$ \\\hline
{$N_{6,19}(\alpha, \beta)$} & {$R_c=c_2(e_{62}-e_{63})+\al c_4e_{64}+\be c_5e_{65}$}, $[R_a,R_b]=(b_1(a_2-a_3)-a_1(b_2-b_3))e_{61}$ & $K^4$ \\\hline
{$N_{6,20}(\alpha, \beta)$} & {$R_c=c_2(e_{62}-e_{63})+\be c_5e_{65}$}, $[R_a,R_b]=(b_1(a_2-a_3)-a_1(b_2-b_3))e_{61}+ \alpha(b_{4}(a_{2} - a_{3}) - a_{4}(b_{2} - b_{3}))e_{64}$ & $K^{4}$ \\\hline
{$N_{6,21}(\alpha, \beta, \gamma)$} & {$R_c=c_2(e_{62}-e_{63})+\g c_5e_{65}$}, $[R_a,R_b]=(b_1(a_2-a_3)-a_1(b_2-b_3))e_{61}+\al (b_4(a_2-a_3)-a_4(b_2-b_3))e_{64}$& $K^4$\\\hline
{$N_{6,22}(\alpha, \beta)$} & {$R_c=c_2(e_{62}-e_{63})$}, $[R_a,R_b]=(b_1(a_2-a_3)-a_1(b_2-b_3))e_{61}+\al (b_4(a_2-a_3)-a_4(b_2-b_3))e_{64}+\be(b_5(a_2-a_3)-a_5(b_2-b_3))e_{65}$ & $K^{4}$ \\\hline
{$N_{6,23}(\alpha, \beta, \gamma, \delta)$} & {$R_c=c_2(e_{62}-e_{63})$}, $[R_a,R_b]=(b_1(a_2-a_3)-a_1(b_2-b_3))e_{61}+\al (b_4(a_2-a_3)-a_4(b_2-b_3))e_{64}+\g(b_5(a_2-a_3)-a_5(b_2-b_3))e_{65}$ & $K^{4}$ \\\hline
{$N_{6,24}(\alpha, \beta, \gamma)$} & {$R_c=c_2(e_{62}-e_{63})$}, $[R_a,R_b]=(b_1(a_2-a_3)-a_1(b_2-b_3))e_{61}+\al (b_4(a_2-a_3)-a_4(b_2-b_3))e_{64}+\g(b_5(a_2-a_3)-a_5(b_2-b_3))e_{65}$& $K^{4}$ \\\hline
\end{longtable}
\end{prop}

\begin{proof} On consid\`{e}re les alg\`{e}bres $N$ suivantes : $N_{4, 6}$ ; $N_{5, 10}(\alpha)$ \`{a} $N_{5, 12}(\alpha, \beta)$ et $N_{6, 19}$ \`{a} $N_{6, 24}(\alpha, \beta, \gamma)$. Comme $e_{1}^{2} = e_{2} + e_{3}$, $e_{2}^2 = e_{n}$, $e_{3}^{2} = -e_{n}$ et $e_{n}^{2} = 0$ avec $n=\dim(N) \geq 4$, alors pour $a, b, c \in N$, on~a $R_{c}(e_{1}) + [R_{a}, R_{b}](e_{1}) = c_{1}e_{1}^{2} + (e_{1}b)a - (e_{1}a)b = c_{1}(e_{2} + e_{3}) + (b_{1}(a_{2} - a_{3}) - a_{1}(b_{2} - b_{3}))e_{n}$, $R_{c}(e_{2}) + [R_{a}, R_{b}](e_{2}) = c_{2}e_{2}^{2} + (e_{2}b)a - (e_{2}a)b = c_{2}e_{n}$, $R_{c}(e_{3}) + [R_{a}, R_{b}](e_{3}) = c_{3}e_{3}^{2} + (e_{3}b)a - (e_{3}a)b = -c_{3}e_{n}$ et $R_{c}(e_{n}) +  [R_{a}, R_{b}](e_{n}) = 0$. On a $-c_3=[R_c]_{n3}=-[R_c]_{n2}=-c_2$ et $c_3=c_2$.

$\bullet$ $N_{4, 6}$ : on a  $R_{c} + [R_{a}, R_{b}] \in \In(N_{4, 6})$ si et seulement si $R_c = c_{2}(e_{42} - e_{43})$ et $[R_{a}, R_{b}] = (b_{1}(a_{2} - a_{3}) - a_{1}(b_{2} - b_{3}))e_{41}$. Ainsi $\In(N_{4, 6}) \simeq K^{2}$.

$\bullet$ $N_{5, 10}(\alpha)$ : on a $R_{c}(e_{4}) + [R_{a}, R_{b}](e_{4}) = \alpha c_{4}e_{5}$ ; donc $R_{c} + [R_{a}, R_{b}] \in \In(N_{5, 10}(\alpha))$ si et seulement si $R_c = c_2(e_{52} - e_{53}) + \alpha c_4e_{54}$ et $[R_{a}, R_{b}] = (b_{1}(a_{2} - a_{3}) - a_{1}(b_{2} - b_{3}))e_{51}$. Ainsi $\In(N_{5, 10}(\alpha)) \simeq K^{3}$.

$\bullet$ $N_{5, 11}(\alpha)$ : on a $R_{c}(e_{4}) + [R_{a}, R_{b}](e_{4}) = \alpha c_{4}(e_{2} + e_{3}) + \alpha(b_{4}(a_{2} - a_{3}) - a_{4}(b_{2} - b_{3}))e_{5}$ ; donc $R_{c} + [R_{a}, R_{b}] \in \In(N_{5, 11}(\alpha))$ si et seulement si $R_c = c_2(e_{52} - e_{53})$ et $[R_{a}, R_{b}] = (b_{1}(a_{2} - a_{3}) - a_{1}(b_{2} - b_{3}))e_{51} + \alpha(b_{4}(a_{2} - a_{3}) - a_{4}(b_{2} - b_{3}))e_{54}$. Ainsi $\In(N_{5, 11}(\alpha)) \simeq K^{3}$.

$\bullet$ $N_{5, 12}(\alpha, \beta)$ : on a $R_{c}(e_{4}) +  [R_{a}, R_{b}](e_{4}) = c_{4}\alpha(e_{2} + e_{3}) + c_{4}\be e_{5} + \alpha(b_{4}(a_{2} - a_{3}) - a_{4}(b_{2} - b_{3}))e_{5}$ ; donc $R_{c} + [R_{a}, R_{b}] \in \In(N_{5, 12}(\alpha, \beta))$ si et seulement si $R_c = c_2(e_{52} - e_{53})$ et $[R_{a}, R_{b}] = (b_{1}(a_{2} - a_{3}) - a_{1}(b_{2} - b_{3}))e_{51} + \alpha(b_{4}(a_{2} - a_{3}) - a_{4}(b_{2} - b_{3}))e_{54}$. Ainsi $\In(N_{5, 12}(\alpha, \beta)) \simeq K^{3}$.

$\bullet$ $N_{6, 19}(\alpha, \beta)$ : on a $R_{c}(e_{4}) + [R_{a}, R_{b}](e_{4}) = \alpha c_{4}e_{6} $ ; $R_{c}(e_{5}) + [R_{a}, R_{b}](e_{5}) = \beta c_{5}e_{6}$ ; donc $R_{c} + [R_{a}, R_{b}] \in \In(N_{6, 19}(\alpha, \beta))$ si et seulement si $R_c = c_2(e_{62} - e_{63}) + \alpha c_4e_{64} + \be c_5e_{65}$ et $[R_{a}, R_{b}] = (b_{1}(a_{2} - a_{3}) - a_{1}(b_{2} - b_{3}))e_{61}$. Ainsi $\In(N_{6, 19}(\alpha, \beta)) \simeq K^{4}$.

$\bullet$ $N_{6, 20}(\alpha, \beta)$ : on a $R_{c}(e_{4}) + [R_{a}, R_{b}](e_{4}) = \alpha c_{4}(e_{2} + e_{3}) + \alpha(b_{4}(a_{2} - a_{3}) - a_{4}(b_{2} - b_{3}))e_{6}$ ; $R_{c}(e_{5}) + [R_{a}, R_{b}](e_{5}) = \be c_{5}e_{6}$ ; donc $R_{c} + [R_{a}, R_{b}] \in \In(N_{6, 20}(\alpha, \beta))$ si et seulement si  $R_c = c_2(e_{62} - e_{63}) + \be c_5e_{65}$ et $[R_{a}, R_{b}] = (b_{1}(a_{2} - a_{3}) - a_{1}(b_{2} - b_{3}))e_{61} + \alpha(b_{4}(a_{2} - a_{3}) - a_{4}(b_{2} - b_{3}))e_{64}$. Ainsi $\In(N_{6, 20}(\alpha, \beta)) \simeq K^{4}$.

$\bullet$ $N_{6, 21}(\alpha, \beta, \gamma)$ : on a $R_{c}(e_{4}) + [R_{a}, R_{b}](e_{4}) = \alpha c_{4}(e_{2} + e_{3}) + \beta c_{4} e_{6} + \alpha(b_{4}(a_{2} - a_{3}) - a_{4}(b_{2} - b_{3}))e_{6}$ ; $R_{c}(e_{5}) + [R_{a}, R_{b}](e_{5}) = \gamma c_{5}e_{6}$; donc $R_{c} + [R_{a}, R_{b}] \in \In(N_{6, 21}(\alpha, \beta, \gamma))$ si et seulement si $R_c = c_2(e_{62} - e_{63}) + \gamma c_5e_{65}$ et $[R_{a}, R_{b}] = (b_{1}(a_{2} - a_{3}) - a_{1}(b_{2} - b_{3}))e_{61} + \alpha(b_{4}(a_{2} - a_{3}) - a_{4}(b_{2} - b_{3}))e_{64}$. Ainsi $\In(N_{6, 21}(\alpha, \beta, \gamma)) \simeq K^{4}$.

$\bullet$ $N_{6, 22}(\alpha, \beta)$ : on a $R_{c}(e_{4}) + [R_{a}, R_{b}](e_{4}) = \alpha c_{4}(e_{2} + e_{3}) +  \alpha(b_{4}(a_{2} - a_{3}) - a_{4}(b_{2} - b_{3}))e_{6}$ ; $R_{c}(e_{5}) + [R_{a}, R_{b}](e_{5}) = \beta c_{5}(e_{2} + e_{3}) + \beta(b_{5}(a_{2} - a_{3}) - a_{5}(b_{2} - b_{3}))e_{6}$ ; donc $R_{c} + [R_{a}, R_{b}] \in \In(N_{6, 22}(\alpha, \beta))$ si et seulement si $R_c = c_2(e_{62} - e_{63})$ et $[R_{a}, R_{b}] = (b_{1}(a_{2} - a_{3}) - a_{1}(b_{2} - b_{3}))e_{61} + \alpha(b_{4}(a_{2} - a_{3}) - a_{4}(b_{2} - b_{3}))e_{64} + \beta(b_{5}(a_{2} - a_{3}) - a_{5}(b_{2} - b_{3}))e_{65}$. Ainsi $\In(N_{6, 22}(\alpha, \beta)) \simeq K^{4}$.

$\bullet$ $N_{6, 23}(\alpha, \beta, \gamma, \delta)$ : on a $R_{c}(e_{4}) + [R_{a}, R_{b}](e_{4}) = \alpha c_{4}(e_{2} + e_{3}) + \beta c_{4}e_{6} + \alpha(b_{4}(a_{2} - a_{3}) - a_{4}(b_{2} - b_{3}))e_{6}$ ; $R_{c}(e_{5}) + [R_{a}, R_{b}](e_{5}) = \gamma c_{5}(e_{2} + e_{3}) + \delta c_{5}e_{6} + \beta(b_{5}(a_{2} - a_{3}) - a_{5}(b_{2} - b_{3}))e_{6}$ ; donc $R_{c} + [R_{a}, R_{b}] \in \In(N_{6, 23}(\alpha, \beta, \gamma, \delta))$ si et seulement si $R_c = c_2(e_{62} - e_{63}) $ et $[R_{a}, R_{b}] = (b_{1}(a_{2} - a_{3}) - a_{1}(b_{2} - b_{3}))e_{61} + \alpha(b_{4}(a_{2} - a_{3}) - a_{4}(b_{2} - b_{3}))e_{64} + \beta(b_{5}(a_{2} - a_{3}) - a_{5}(b_{2} - b_{3}))e_{65}$. Ainsi $\In(N_{6, 23}(\alpha, \beta, \gamma, \delta)) \simeq K^{4}$.

$\bullet$ $N_{6, 24}(\alpha, \beta, \gamma)$ : on a $R_{c}(e_{4}) + [R_{a}, R_{b}](e_{4}) = \alpha c_{4}(e_{2} + e_{3}) + \beta c_{4}e_{6} + \alpha(b_{4}(a_{2} - a_{3}) - a_{4}(b_{2} - b_{3}))e_{6}$ ; $R_{c}(e_{5}) + [R_{a}, R_{b}](e_{5}) = \gamma c_{5}(e_{2} + e_{3}) + \gamma(b_{5}(a_{2} - a_{3}) - a_{5}(b_{2} - b_{3}))e_{6}$ ; donc $R_{c} + [R_{a}, R_{b}] \in \In(N_{6, 24}(\alpha, \beta, \gamma))$ si et seulement si $R_c = c_2(e_{62} - e_{63})$ et $[R_{a}, R_{b}] = (b_{1}(a_{2} - a_{3}) - a_{1}(b_{2} - b_{3}))e_{61} + \alpha(b_{4}(a_{2} - a_{3}) - a_{4}(b_{2} - b_{3}))e_{64} + \gamma(b_{5}(a_{2} - a_{3}) - a_{5}(b_{2} - b_{3}))e_{65}$. Ainsi $\In(N_{6, 24}(\alpha, \beta, \gamma)) \simeq K^{4}$.
\end{proof}

\subsection{Dimension de l'annulateur de $N$ est $2$}

\begin{theo} Soit $N$ une nil-alg\`{e}bre d'\'{e}volution ind\'{e}composable \`{a} puissances associatives de dimension $\leq 6$, qui n'est pas associative, telle que $\dim(ann(N)) = 2$. Alors, la Table~$10$ caract\'{e}rise les d\'{e}rivations dans $N$.

\begin{longtable}[c]{|p{2cm}|p{9.5cm}|c|}
 \caption*{\underline{\textbf{Table 10}} }\\
  \hline
    {$N$} &   {\small D\'{e}finition de la d\'{e}rivation} &   {\small $\D(N)\simeq$}   \\\hline
\endfirsthead
 \hline   {$N$} &   {\small D\'{e}rivation} &   {\small $\D(N)\simeq$}   \\\hline
\endhead \hline
\endfoot
{$N_{6,25}(\alpha)$} & {$d(e_{1}) = d_{11}e_{1} + d_{51}e_{5} + d_{61}e_{6}$, $d(e_{2}) = d_{22}e_{2} + (2d_{11} - d_{22})e_{3} + d_{52}e_{5} + d_{62}e_{6}$, $d(e_{3}) = (2d_{11} - d_{22})e_{2} + d_{22}e_{3} - d_{52}e_{5} - d_{62}e_{6}$, $d(e_{4}) = d_{11}e_{4} + d_{54}e_{5} + d_{64}e_{6}$, $d(e_{5}) = 2d_{22}e_{5}$, $d(e_{6}) = 2d_{11}e_{6}$} & {$K^{5}\times K^3$}   \\\hline
{$N_{6,26}$} & {$d(e_{1}) = d_{11}e_{1} + d_{51}e_{5} + d_{61}e_{6}$, $d(e_{2}) = 2d_{11}e_{2} + d_{52}e_{5} + d_{62}e_{6}$, $d(e_{3}) = 2d_{11}e_{3} + d_{53}e_{5} + d_{63}e_{6}$, $d(e_{4}) = 2d_{11}e_{4} -(d_{52} + d_{53})e_{5} - (d_{62} + d_{63})e_{6}$, $d(e_{5}) = 4d_{11}e_{5}$, $d(e_{6}) = 4d_{11}e_{6}$, } & {$K^{4}\times K^3$}   \\\hline
\end{longtable}
\end{theo}

\begin{proof}

$\bullet$ $N_{6, 25}(\alpha)$ : $e_{1}^{2} =  e_{2} + e_{3}$, $e_{2}^{2} = e_{5}$, $e_{3}^{2} = -e_{5}$, $e_{4}^{2} = \alpha(e_{2} + e_{3}) + e_{6}$, $e_{5}^{2} = e_{6}^{2} = 0$ avec $\alpha \neq 0$. $(a)$ entra\^{\i}ne $d_{14} = d_{41}$, $d_{24} = d_{42} = 0$, $d_{34} = d_{43} = 0$ et comme $e_{1}^{2} =  e_{2} + e_{3}$, $e_{2}^{2} = e_{5}$, $e_{3}^{2} = -e_{5}$, $e_{5}^{2} = 0$, alors \eqref{Eq6} donne $d(e_{1}) = d_{11}e_{1} + d_{51}e_{5} + d_{61}e_{6}$, $d(e_{2}) = d_{22}e_{2} + (2d_{11} - d_{22})e_{2} + d_{52}e_{5} + d_{62}e_{6}$, $d(e_{3}) = (2d_{11} - d_{22})e_{2} + d_{22}e_{2} - d_{52}e_{5} - d_{62}e_{6}$ et $d(e_{5}) = 2d_{22}e_{5}$. En d\'{e}rivant $e_{4}^{2} = \alpha e_{1}^{2} + e_{6}$, on obtient $2\alpha d_{11}e_{1}^{2} + d(e_{6}) = 2d_{44}e_{4}^{2} = 2\alpha d_{44}(e_{2} + e_{3}) + 2d_{44}e_{6}$. Alors $d(e_{6}) = 2\alpha(d_{44} - d_{11})(e_{2} + e_{3}) + 2d_{44}e_{6}$, donc $d_{44} = d_{11}$ car $d(e_{6}) \in ann(N)$. Ainsi  $d(e_{4}) = d_{11}e_{4} + d_{54}e_{5} + d_{64}e_{6}$ et $d(e_{6}) = 2d_{11}e_{6}$.

$\bullet$ $N_{6, 26}$ : $e_{1}^{2} =  e_{2} + e_{3} + e_{4}$, $e_{2}^{2} = e_{5}$, $e_{3}^{2} = e_{6}$, $e_{4}^{2} = -(e_{5} + e_{6})$, $e_{5}^{2} = e_{6}^{2} = 0$. $(a)$ entra\^{\i}ne $d_{12} = d_{21} = 0$, $d_{13} = d_{31} = 0$, $d_{14} = d_{41} = 0$, $d_{23} = d_{32} = 0$, $d_{24} = d_{42} = 0$ et $d_{34} = d_{43} = 0$. En d\'{e}rivant $e_{2}^{2} = e_{5}$ et $e_{3}^{2} = e_{6}$ respectivement, on obtient $d(e_{5}) = 2d_{22}e_{5}$ et $d(e_{6}) = 2d_{33}e_{6}$. En d\'{e}rivant $e_{4}^{2} = -(e_{5} + e_{6})$, on obtient $-(d(e_{5}) + d(e_{6})) = 2d_{44}e_{4}^{2} = -2d_{44}(e_{5} + e_{6})$, i.e. $2(d_{44} - d_{22})e_{5} + 2(d_{44} - d_{33})e_{6} = 0$ ;  donc $d_{22} = d_{44} = d_{33}$. En d\'{e}rivant $e_{1}^{2} = e_{2} + e_{3} + e_{4}$, on obtient $d(e_{2}) + d(e_{3}) + d(e_{4}) = 2d_{11}e_{1}^{2} = 2d_{11}(e_{2} + e_{3} + e_{4})$, i.e. $d_{22}(e_{2} + e_{3} + e_{4}) + (d_{52} + d_{53} + d_{54})e_{5} + (d_{62} + d_{63} + d_{64})e_{6}$ ; donc $d_{22} = 2d_{11}$, $d_{54} = -(d_{52} + d_{53})$ et $d_{64} = -(d_{62} + d_{63})$. Ainsi $d(e_{1}) = d_{11}e_{1} + d_{51}e_{5} + d_{61}e_{6}$, $d(e_{2}) = 2d_{11}e_{2} + d_{52}e_{5} + d_{62}e_{6}$, $d(e_{3}) = 2d_{11}e_{3} + d_{53}e_{5} + d_{63}e_{6}$, $d(e_{4}) = 2d_{11}e_{4} - (d_{52} + d_{53})e_{5} - (d_{62} + d_{63})e_{6}$, $d(e_{5}) = 4d_{11}e_{5}$, $d(e_{6}) = 4d_{11}e_{6}$.
\end{proof}

\begin{prop} Soit $N$ une nil-alg\`{e}bre d'\'{e}volution ind\'{e}composable \`{a} puissances associatives de dimension $\leq 6$, qui n'est pas associative, telle que $\dim(ann(N)) = 2$. Alors le crochet des d\'{e}rivations $d$ et $d'$ de $N$ est donn\'{e} dans la Table~$11$.

\begin{longtable}[c]{|p{2cm}|p{9cm}|c|}
 \caption*{\underline{\textbf{Table 11}
} }\\
  \hline
    {$N$} &   {\small D\'{e}finition de la d\'{e}rivation $[d, d']$} &   {\small $\D(N)'\simeq$}   \\\hline
\endfirsthead
 \hline   {$N$} &   {\small D\'{e}finition de la d\'{e}rivation $[d, d']$} &   {\small $\D(N)'\simeq$}   \\\hline
\endhead \hline
{$N_{6,25}(\alpha, \beta)$} & {$[d, d']_{11} = [d, d']_{22} = 0$, $[d, d']_{j1} = d_{j1}'d_{11} - d_{j1}d_{11}'$, $[d, d']_{j2} = 2(d_{j2}'d_{11} - d_{j2}d_{11}')$, $[d, d']_{54} = (d_{11}'d_{54} - d_{11}d_{54}') + 2(d_{54}'d_{22} - d_{54}d_{22}')$, $[d, d']_{64} = d_{64}'d_{11} - d_{64}d_{11}'$ avec $j = 5, 6$} & {$K^{3}\times K^3$}   \\\hline
{$N_{6,26}$} & {$[d, d']_{11} = 0$, $[d, d']_{j1} = 3(d_{j1}'d_{11} - d_{j1}d_{11}')$, $[d, d']_{ji} = 2(d_{ji}'d_{11} - d_{ji}d_{11}')$ avec $j = 5, 6$ et $i = 2, 3$} & {$K^{3}\times K^3$}   \\\hline
\end{longtable}
\end{prop}

\begin{proof}Soient $d, d' \in \D(N)$. Calculons $[d,d']$.

$\bullet$ $N_{6, 25}(\alpha)$: On a $d\circ d' (e_{1}) = d_{11}'d(e_{1}) + d_{51}'d(e_{5}) + d_{61}'d(e_{6})$, alors $[d, d']_{11} = d_{11}'d_{11} - d_{11}d_{11}' = 0$, $[d, d']_{51} = (d_{11}'d_{51} + 2d_{51}'d_{22}) - (d_{11}d_{51}' + 2d_{51}d_{22}') = d_{51}'d_{22} - d_{51}d_{22}'$, $[d, d']_{61} = (d_{11}'d_{61} + 2d_{61}'d_{11}) - (d_{11}d_{61}' + 2d_{61}d_{11}') = d_{61}'d_{11} - d_{61}d_{11}'$. On a $d\circ d' (e_{2}) = d_{22}'d(e_{2}) + (2d_{11}' - d_{22}')d(e_{3}) + d_{52}'d(e_{5}) + d_{62}'d(e_{6})$, alors $[d, d']_{22} = (d_{22}'d_{22} + (2d_{11}' - d_{22}')(2d_{11} - d_{22}) - (d_{22}d_{22}' + (2d_{11} - d_{22})(2d_{11}' - d_{22}') = 0$, $[d, d']_{52} = (d_{22}'d_{52} - (2d_{11}' - d_{22}')d_{52} + 2d_{52}'d_{22}) - (d_{22}d_{52}' - (2d_{11} - d_{22})d_{52}' + 2d_{52}d_{22}') = 2(d_{11}d_{52}' - d_{11}'d_{52})$, $[d, d']_{62} = (d_{22}'d_{62} - (2d_{11}' - d_{22}')d_{62} + 2d_{62}'d_{11}) - (d_{22}d_{62}' - (2d_{11} - d_{22})d_{62}' + 2d_{62}d_{11}') = 2(d_{62}'d_{11} - d_{62}d_{11}')$. On a $d\circ d' (e_{4}) = d_{11}'d(e_{4}) + d_{54}'d(e_{5}) + d_{64}'d(e_{6})$, alors $[d, d']_{54} = (d_{11}'d_{54} + 2d_{54}'d_{22}) - (d_{11}d_{54}' + 2d_{54}d_{22}') = (d_{11}'d_{54} - d_{11}d_{54}') + 2(d_{54}'d_{22} - d_{54}d_{22}')$, $[d, d']_{64} = (d_{11}'d_{64} + 2d_{64}'d_{11}) - (d_{11}d_{64}' + 2d_{64}d_{11}') = d_{64}'d_{11} - d_{64}d_{11}'$.
On en d\'{e}duit que  $\D(N_{6, 25}(\alpha))'\simeq K^{6}$.

$\bullet$ $N_{6, 26}$: On a $d\circ d' (e_{1}) = d_{11}'d(e_{1}) + d_{51}'d(e_{5}) + d_{61}'d(e_{6})$, alors $[d, d']_{11} = d_{11}'d_{11} - d_{11}d_{11}' = 0$, $[d, d']_{51} = (d_{11}'d_{51} + 4d_{51}'d_{11}) - (d_{11}d_{51}' + 4d_{51}d_{11}') = 3(d_{51}'d_{11} - d_{51}d_{11}')$, $[d, d']_{61} = (d_{11}'d_{61} + 4d_{61}'d_{11}) - (d_{11}d_{61}' + 4d_{61}d_{11}') = 3(d_{61}'d_{11} - d_{61}d_{11}')$. On a $d\circ d' (e_{2}) = 2d_{11}'d(e_{2}) + d_{52}'d(e_{5}) + d_{62}'d(e_{6})$, alors $[d, d']_{52} = (2d_{11}'d_{52} + 4d_{52}'d_{11}) - (2d_{11}d_{52}' + 4d_{52}d_{11}') = 2(d_{52}'d_{11} - d_{52}d_{11}')$, $[d, d']_{62} = (2d_{11}'d_{62} + 4d_{62}'d_{11}) - (2d_{11}d_{62}' + 4d_{62}d_{11}') = 2(d_{62}'d_{11} - d_{62}d_{11}')$. On a $d\circ d' (e_{3}) = 2d_{11}'d(e_{3}) + d_{53}'d(e_{5}) + d_{63}'d(e_{6})$, alors $[d, d']_{53} = (2d_{11}'d_{53} + 4d_{53}'d_{11}) - (2d_{11}d_{53}' + 4d_{53}d_{11}') = 2(d_{53}'d_{11} - d_{53}d_{11}')$, $[d, d']_{63} = (2d_{11}'d_{63} + 4d_{63}'d_{11}) - (2d_{11}d_{63}' + 4d_{63}d_{11}') = 2(d_{63}'d_{11} - d_{63}d_{11}')$. On en d\'{e}duit que $\D(N_{6, 26})'\simeq K^{6}$.
\end{proof}

\begin{prop} Soit $N$ une nil-alg\`{e}bre d'\'{e}volution ind\'{e}composable \`{a} puissances associatives et qui n'est pas associative de dimension $\leq 6$ telle que $\dim(ann(N)) = 2$. La Table~$12$ caract\'{e}rise les d\'{e}rivations int\'{e}rieures de $N$.

\begin{longtable}[c]{|p{2.5cm}|p{9cm}|c|}
 \caption*{\underline{\textbf{Table 12}} }\\
  \hline
    {$N$} &   {\small $R_{c}+[R_a,R_b] \in \D(N)$ si et seulement si} &   {\small $\In(N) \simeq$}   \\\hline
\endfirsthead
 \hline   {$N$} &   {\small $R_{c}+[R_a,R_b] \in \D(N)$ si et seulement si} &   {\small $\In(N) \simeq$}   \\\hline
\endhead \hline
{$N_{6, 25}(\alpha)$} & {$R_c=c_2(e_{52}-e_{53})$}, $[R_a,R_b]=(b_1(a_2-a_3)-a_1(b_2-b_3))e_{51}+\al (b_4(a_2-a_3)-a_4(b_2-b_3))e_{54}$ & {$K^{3}$}   \\\hline
{$N_{6, 26}$} & {$R_c = c_2(e_{52} + e_{63} - e_{54} - e_{64})$}, $[R_a,R_b]=(b_1(a_2-a_4)-a_1(b_2-b_4))e_{51}+\al (b_1(a_3-a_4)-a_1(b_3-b_4))e_{61}$ & {$K^3$}   \\\hline
\end{longtable}
\end{prop}

\begin{proof}
$\bullet$ $N_{6, 25}(\alpha)$ : comme $e_{1}^{2} = e_{2} + e_{3}$, $e_{2}^{2} = e_{5}$, $e_{3}^{2} = -e_{5}$ $e_{5}^{2} = 0$ alors, pour $a, b, c \in N$, on~a $R_{c}(e_{1})=c_1(e_2+e_3)$, $R_c(e_2)=c_2e_5$, $R_c(e_3)=-c_3e_5$, $R_c(e_4)=c_4(\al(e_2+e_3)+e_6)$, $R_c(e_5)=R_c(e_6)=0$. Comme $c_2=[R_c]_{52}$, $[R_c]_{53}=-[R_c]_{52}=-c_3$ alors $c_3=c_2$ et $R_c=c_2(e_{52}-e_{53})$. Par ailleurs $[R_{a}, R_{b}](e_{1}) = (b_{1}(a_{2} - a_{3}) - a_{1}(b_{2} - b_{3}))e_{5}$, $[R_{a}, R_{b}](e_{2}) =[R_{a}, R_{b}](e_{3})=0$, $[R_{a}, R_{b}](e_{4}) = \al(b_{4}(a_{2} - a_{3}) - a_{4}(b_{2} - b_{3}))e_{5}$ et $[R_{a}, R_{b}](e_{5}) =[R_{a}, R_{b}](e_{6}) = 0$ ; donc $R_{c} + [R_{a}, R_{b}] \in \In(N_{6, 25}(\alpha))$ si et seulement si  $R_c = c_2(e_{52} - e_{53})$ et $[R_{a}, R_{b}] = (b_{1}(a_{2} - a_{3}) - a_{1}(b_{2} - b_{3}))e_{51} + \alpha(b_{4}(a_{2} - a_{3}) - a_{4}(b_{2} - b_{3}))e_{54}$. Ainsi $\In(N_{6, 25}(\alpha)) \simeq K^{3}$.

$\bullet$ $N_{6, 26}$ : on a $R_{c}(e_{1}) = c_{1}e_{1}^{2} = c_{1}(e_{2} + e_{3} + e_{4})$, $R_c(e_2)=c_2e_5$, $R_c(e_3)=c_3e_6$, $R_c(e_4)=-c_4(e_5+e_6)$, $R_c(e_5)=R_c(e_6)=0$. On a $c_1=[R_c]_{21}=0$, $c_2= [R_c]_{52}$, $c_3=[R_c]_{63}$, $-c_4=[R_c]_{54}=[R_c]_{64}$ donne $-c_4=-([R_c]_{52}+[R_c]_{53})= -([R_c]_{62}+[R_c]_{63})=-[R_c]_{6,3}$ et $-c_4=-c_3=-[R_c]_{52}=-c_2$, par suite $R_c=c_2(e_{52} + e_{63} - e_{54} - e_{64})$. Par ailleurs $[R_{a}, R_{b}](e_{1})= (e_{1}b)a - (e_{1}a)b= (b_{1}(a_{2} - a_{4}) - a_{1}(b_{2} - b_{4}))e_{5} + (b_{1}(a_{3} - a_{4}) - a_{1}(b_{3} - b_{4}))e_{6}$, $[R_{a}, R_{b}](e_{i}) = 0$ pour $i=2,3,4,5,6$.  Donc $R_{c} + [R_{a}, R_{b}]$ est une d\'{e}rivation de $N_{6, 26}$ si et seulement si $[R_c]=c_2(e_{52} + e_{63} - e_{54} - e_{64})$ et $[R_{a}, R_{b}] = (b_{1}(a_{2} - a_{4}) - a_{1}(b_{2} - b_{4}))e_{51} + (b_{1}(a_{3} - a_{4}) - a_{1}(b_{3} - b_{4}))e_{61}$. Ainsi $\In(N_{6, 26}) \simeq K^{2}$.

\end{proof}

\end{document}